\documentclass[reqno]{amsart}
\usepackage{amsmath, amsthm, amssymb, amstext, bbm}

\usepackage{hyperref,xcolor}
\hypersetup{pdfborder={0 0 0},colorlinks}
\usepackage{enumitem}
\allowdisplaybreaks
\newtheorem{theorem}{Theorem}
\newtheorem{remark}[theorem]{Remark}
\newtheorem{lemma}[theorem]{Lemma}
\newtheorem{proposition}[theorem]{Proposition}

\newtheorem{definition}[theorem]{Definition}
\newtheorem{example}[theorem]{Example}

\DeclareMathOperator*{\Ss}{S}
\newcommand*\diff{\mathrm{d}}
\newcommand{\N}{\mathbb{N}}

\newcommand{\R}{\mathbb{R}}

\newcommand{\WH}{W^{1, \mathcal{H}}(\Omega)}
\newcommand{\norm}[1]{\|#1\|}

\newcommand{\Lp}[1]{L^{#1}(\Omega)}
\newcommand{\Lprand}[1]{L^{#1}(\partial\Omega)}
\newcommand{\Wp}[1]{W^{1,#1}(\Omega)}

\newcommand{\lan}{\langle}
\newcommand{\ran}{\rangle}
\newcommand{\eps}{\varepsilon}
\newcommand{\into}{\int_{\Omega}}
\newcommand{\intor}{\int_{\partial\Omega}}
\newcommand{\weak}{\rightharpoonup}

\newcommand{\Linf}{L^{\infty}(\Omega)}
\newcommand{\close}{\overline{\Omega}}
\renewcommand{\l}{\left}
\renewcommand{\r}{\right}

\numberwithin{theorem}{section}
\numberwithin{equation}{section}

\title[Superlinear elliptic equations with unbalanced growth]
{Superlinear elliptic equations with unbalanced growth and nonlinear boundary condition}

\author[E. Amoroso]{Eleonora Amoroso}
\address[E. Amoroso]{Department of Engineering, University of Messina, 98166 Messina, Italy}
\email{eleonora.amoroso@unime.it}

\author[\'{A}. Crespo-Blanco]{\'{A}ngel Crespo-Blanco}
\address[\'{A}. Crespo-Blanco]{Technische Universit\"{a}t Berlin, Institut f\"{u}r Mathematik, Stra\ss e des 17.\,Juni 136, 10623 Berlin, Germany}
\email{crespo@math.tu-berlin.de}

\author[P. Pucci]{Patrizia Pucci}
\address[P. Pucci]{Department of Mathematics and Computer Science, University of Perugia, Via Vanvitelli 1, 06123 Perugia, Italy}
\email{patrizia.pucci@unipg.it}

\author[P. Winkert]{Patrick Winkert}
\address[P. Winkert]{Technische Universit\"{a}t Berlin, Institut f\"{u}r Mathematik, Stra\ss e des 17.\,Juni 136, 10623 Berlin, Germany}
\email{winkert@math.tu-berlin.de}

\begin{document}

\begin{abstract}
	In this paper we first introduce an innovative equivalent norm in the Musielak-Orlicz Sobolev spaces in a very general setting and  we then present a new result on the boundedness of the solutions of a wide class of nonlinear Neumann problems, both of independent interest. Moreover, we study a variable exponent double phase problem with a nonlinear boundary condition and prove the existence of multiple solutions under very general assumptions on the nonlinearities. To be more precise, we get constant sign solutions  (nonpositive and nonnegative)  via a mountain-pass  approach and a sign-changing solution by using an appropriate subset of the corresponding Nehari manifold along with the Brouwer degree and the Quantitative Deformation Lemma.
\end{abstract}


\maketitle

\section{Introduction}
A differential operator that has found a place in many research fields in recent years is the so-called ``double phase operator", which is defined by
\begin{align*}
	u \mapsto -\operatorname{div}\l(|\nabla u|^{p(x)-2}\nabla u + \mu(x)|\nabla u|^{q(x)-2}\nabla u\r),
\end{align*}
for every function $u$ belonging to a suitable Musielak-Orlicz Sobolev space $\WH$, where $\Omega\subset\R^N$ is supposed to be a bounded domain with Lipschitz boundary $\partial\Omega$. The integral functional related to this operator, given by
\begin{align*}
	\into \l( |\nabla u|^{p(x)} + \mu(x) |\nabla u|^{q(x)} \r) \,\diff x, \quad u \in \WH,
\end{align*}
changes ellipticity in two different phases and has been first introduced in 1986 by Zhikov \cite{Zhikov-1986} with constant exponents. Since then, many authors studied problems involving this operator, which has been used to model different phenomena. Among the topics, we mention first the elasticity theory in which it describes the behavior of strongly anisotropic materials, whose hardening properties are related to the exponents $p(\cdot)$ and $q(\cdot)$ and significantly change with the point and the coefficient $\mu(\cdot)$ determines the geometry of a composite made of two different materials, see Zhikov \cite{Zhikov-2011}. Moreover, other applications can be found in the works of Bahrouni-R\u{a}dulescu-Repov\v{s} \cite{Bahrouni-Radulescu-Repovs-2019} on transonic flows, Benci-D’Avenia-Fortunato-Pisani \cite{Benci-DAvenia-Fortunato-Pisani-2000} on quantum physics and Zhikov \cite{Zhikov-2011} on the Lavrentiev gap phenomenon, the thermistor problem and the duality theory. For a mathematical study of such integral functionals with $(p,q)$-growth we refer to the works of Baroni-Colombo-Mingione \cite{Baroni-Colombo-Mingione-2015, Baroni-Colombo-Mingione-2016, Baroni-Colombo-Mingione-2018}, Colombo-Mingione \cite{Colombo-Mingione-2015a,Colombo-Mingione-2015b}, Cupini-Marcellini-Mascolo \cite{Cupini-Marcellini-Mascolo-2023}, De Filippis-Mingione \cite{DeFilippis-Mingione-2023}, Marcellini \cite{Marcellini-2023,Marcellini-1991,Marcellini-1989b}, Ragusa-Tachikawa \cite{Ragusa-Tachikawa-2020}, see also the papers of Beck-Mingione \cite{Beck-Mingione-2020} and De Filippis-Mingione \cite{DeFilippis-Mingione-2021} for nonautonomous integrals. Furthermore  it should be mentioned that the double phase operator generalizes several other differential operators, for example, the $(p(\cdot),q(\cdot))$- Laplacian when $\inf_{\close} \mu >0$ and the $p(\cdot)$-Laplacian if $\mu \equiv 0$, respectively, both of which have been extensively studied in the literature.

Concerning applications in partial differential equations, the double phase operator arises from the study of general reaction--diffusion equations with nonhomogeneous diffusion and transport aspects. These nonhomogeneous operators have applications in biophysics, plasma physics and chemical reactions, with double phase features, where the function $u$ corresponds to the concentration term, and the differential operator represents the diffusion coefficient.

The weak solutions of related problems are functions belonging to an appropriate Musielak-Orlicz Sobolev space $\WH$, where $\mathcal{H}\colon \Omega \times [0,\infty)\to [0,\infty)$ is a nonlinear function defined by
\begin{align*}
	(x,t) \mapsto t^{p(x)}+\mu(x)t^{q(x)},
\end{align*}
with
\begin{align*}
	1 < p(x) < N, \quad
	p(x) < q(x) < p^*(x) = \frac{Np(x)}{N-p(x)} \quad \text{for all } x \in \close,
\end{align*}
and $0 \le \mu(\cdot) \in \Linf$. The novelties of our paper can be summarized as follows and affect different results of independent interest. First, we prove the existence of a new general equivalent norm on $\WH$ given by
\begin{align}\label{eqivalent-norm-new}
	\begin{split}
		\|u\|_{1,\mathcal{H}}^* = \inf \Bigg\{ \tau>0 \,:\, & \into \l(\l(\frac{|\nabla u|}{\tau}\r)^{p(x)}+\mu(x)\l(\frac{|\nabla u|}{\tau}\r)^{q(x)}\r)\,\diff x \\
		& +\into \vartheta_1(x) \l(\frac{|u|}{\tau}\r)^{\delta_1(x)}\,\diff x\\
		& +\intor\vartheta_2(x)\l(\frac{|u|}{\tau}\r)^{\delta_2(x)}\,\diff \sigma \leq 1  \Bigg\},
	\end{split}
\end{align}
where $0 \le \vartheta_1(\cdot) \in  \Linf, 0 \le \vartheta_2(\cdot) \in L^\infty(\partial\Omega)$ and $\delta_1, \delta_2$ are of class $C(\close)$, satisfying $1\leq \delta_1(x)\leq p^*(x)$ and $1\leq \delta_2(x) \leq p_*(x)$ for all $x\in\close$. For more details we refer to \eqref{H1}. Let us emphasize that in this setting the exponents $\delta_1(\cdot)$ and $\delta_2(\cdot)$ can be also critical, namely they can coincide (at some points or at all ones) with the Sobolev critical exponents $p^*(\cdot)$ and $p_*(\cdot)$, respectively, see \eqref{def_critical_exp} for the definition of them. There is a trade-off for allowing the exponents $\delta_1$ and $\delta_2$ to be equal to the Sobolev critical exponents at some points: it does not suffice that $\delta_1$ and $\delta_2$ are continuous functions, we require that they are log-H\"{o}lder continuous and in $W^{1,\gamma} (\Omega)$ for $\gamma>N$, respectively. The reason is the Sobolev embedding theorem in variable exponent spaces, which requires this extra regularity if you achieve equality with the critical Sobolev exponent. Note that in the constant exponent case there would be no difference. All in all, the norm in \eqref{eqivalent-norm-new} generalizes different known norms in $W^{1,p}(\Omega)$, $W^{1,p(\cdot)}(\Omega)$ or in the Musielak-Orlicz Sobolev space with constant exponents, see Crespo-Blanco-Papageorgiou-Winkert \cite{Crespo-Blanco-Papageorgiou-Winkert-2022}.

In the second part of the paper, we are interested in the boundedness of weak solutions of the following nonlinear Neumann problem
\begin{equation}\label{pr}
	\begin{aligned}
		-\operatorname{div} \mathcal{A}(x,u,\nabla u)
		 & = \mathcal{B}(x,u,\nabla u) \quad &  & \text{in } \Omega, \\
		\mathcal{A}(x,u,\nabla u) \cdot \nu
		 & = \mathcal{C} (x,u) &  & \text{on } \partial\Omega,
	\end{aligned}
\end{equation}
where the right-hand side in $\Omega$ can also depend on the gradient of the solution and $\mathcal{A}, \mathcal{B}$ and $\mathcal{C}$ are Carath\'eodory functions satisfying suitable and general growth conditions presented in \eqref{H_infty}. In particular, any weak solution of \eqref{pr} turns out to be in $\Linf$ and we give in Theorem \ref{bounded-solutions} a priori estimates on its $\Linf$-norm. Such a result can be applied in several other problems involving the variable exponent double phase operator as well as general right-hand sides.

In the last part of this paper our purpose is to prove existence and multiplicity results for a variable exponent double phase problem with nonlinear boundary condition and superlinear nonlinearities. Inspired by the recent work of Crespo-Blanco-Winkert \cite{Crespo-Blanco-Winkert-2022} on a Dirichlet problem, the new equivalent norm that we present in Section \ref{Section:equivalent norm} plays an important role. In particular, given a bounded domain $\Omega\subset \R^N$, $N\geq 2$, with Lipschitz boundary $\partial \Omega$ and denoting with $\nu(x)$ the outer unit normal of $\Omega$ at $x \in \partial\Omega$, we study the following problem
\begin{equation}\label{problem}\tag{$P$}
	\begin{aligned}
		-\operatorname{div} \mathcal{F}(u)+ |u|^{p(x)-2}u &= f(x,u) &  &\text{in } \Omega,\\
		\mathcal{F}(u) \cdot  \nu &= g(x,u) - |u|^{p(x)-2}u     &   & \text{on } \partial\Omega,
	\end{aligned}
\end{equation}
where $\operatorname{div} \mathcal{F}(u)$ is the variable exponent double phase operator given by
\begin{align*}
	\mathcal{F}(u):=|\nabla u|^{p(x)-2}\nabla u+\mu(x) |\nabla u|^{q(x)-2}\nabla u,
\end{align*}
and $f\colon\Omega \times \R\to\R$ as well as $g\colon\partial\Omega \times \R\to\R$ are Carath\'eodory functions which are superlinear with respect to the second argument, see the precise conditions in \eqref{H_{f,g}} and some examples in Example \ref{examples}.

In recent years, many authors have dealt with double phase problems in the constant exponents case, see for instance Biagi-Esposito-Vecchi \cite{Biagi-Esposito-Vecchi-2021}, Colasuonno-Squassina \cite{Colasuonno-Squassina-2016}, Farkas-Winkert \cite{Farkas-Winkert-2021}, Fiscella \cite{Fiscella-2022}, Gasi\'nski-Papageorgiou \cite{Gasinski-Papageorgiou-2021}, Gas\-i\'nski-Winkert \cite{Gasinski-Winkert-2020b, Gasinski-Winkert-2021}, Ge-Pucci \cite{Ge-Pucci-2022}, Liu-Dai \cite{Liu-Dai-2018}, Liu-Papageorgiou \cite{Liu-Papageorgiou-2022}, Papageor\-giou-R\u{a}dulescu-Repov\v{s} \cite{Papageorgiou-Radulescu-Repovs-2019b}, Perera-Squassina \cite{Perera-Squassina-2018}, Pucci \cite{Pucci-2023}, Stegli\'{n}ski \cite{Steglinski-2022}, Zeng-Bai-Gasi\'n\-ski-Winkert \cite{Zeng-Bai-Gasinski-Winkert-2020} and the references therein.

On the other hand, there are much fewer results for the variable exponents case, see Amoroso-Bonanno-D'Agu\`\i-Winkert \cite{Amoroso-Bonanno-DAgui-Winkert}, Bahrouni-R\u{a}dulescu-Winkert \cite{Bahrouni-Radulescu-Winkert-2020}, Crespo-Blanco-Gasi\'nski-Harjulehto-Winkert \cite{Crespo-Blanco-Gasinski-Harjulehto-Winkert-2022}, Crespo-Blanco-Winkert \cite{Crespo-Blanco-Winkert-2022}, Leo\-nardi-Papageorgiou \cite{Leonardi-Papageorgiou-2023}, Liu-Pucci \cite{Liu-Pucci-2023},  Kim-Kim-Oh-Zeng \cite{Kim-Kim-Oh-Zeng-2022}, Ragusa-Tachikawa \cite{Ragusa-Tachikawa-2020}, Vetro-Winkert \cite{Vetro-Winkert-2023} and Zeng-R\u{a}dulescu-Winkert \cite{Zeng-Radulescu-Winkert-2022}.

As mentioned before, we present existence and multiplicity results for problem \eqref{problem} by using critical point theory and the Nehari manifold approach, that is, we are able to provide the existence of three bounded weak solutions of problem \eqref{problem} with precise information on the sign. Indeed, through a mountain-pass approach we obtain the existence of two solutions with constant sign. In addition, through the Nehari manifold method along with the Quantitative Deformation Lemma and the Brouwer degree we establish the existence of a sign-changing solution, that turns out to have exactly two nodal domains. We emphasize that we do not require a monotonicity condition on the exponent $p(\cdot)$ as it was needed in the work of Crespo-Blanco-Winkert \cite[hypothesis (H1)]{Crespo-Blanco-Winkert-2022}, since we do not need Poincar\`e's inequality for the modular function related to the norm. Moreover, as far as we know, the growth assumption on the boundary $\partial\Omega$ stated in \eqref{h_4} is new and optimal for this treatment.

The paper is organized as follows. In Section \ref{Section:preliminaries} we recall the definitions and some pro\-per\-ties of the Lebesgue and Sobolev spaces with variable exponents and of the Musielak-Orlicz Sobolev spaces as well as the main tools needed in our treatment, such as the  Mountain-Pass Theorem  (Theorem \ref{mountain_pass_theorem}) and the Quantitative Deformation Lemma (Lemma \ref{quantitative_deformation_lemma}).
In Section~\ref{Section:equivalent norm} we present the proof of a new equivalent norm in the Musielak-Orlicz Sobolev space and we give some properties related to the corresponding modular and the operator. In Section \ref{Section:bounded solutions} we provide a result on the boundedness of the weak solutions of a more general problem than \eqref{problem}, giving also in
Theorem \ref{bounded-solutions} a priori estimates on the $\Linf$-norm of the weak solutions. Then, in Section \ref{Section:constant sign} we state the assumptions on the nonlinearities $f$ and $g$ and in Theorem \ref{theorem_2_sol}
we prove the existence of two constant sign solutions, in particular one is nonnegative and the other one is nonpositive. After this, in Section \ref{Section:sign changing} we state  Theorem \ref{theorem_3_sol} concerning the existence of a third solution, which is sign-changing, obtained minimizing the energy functional related to our problem in a suitable Nehari manifold subset. Finally, Theorem \ref{theorem_nodal_domains} gives information on the nodal domains of this sign-changing solution.

\section{Preliminaries}\label{Section:preliminaries}
For any $1\leq r\leq \infty$, $\Lp{r}$ indicates the usual Lebesgue spaces equipped with the norm $\|\cdot\|_r$ and for $1\leq r < \infty$, $\Wp{r}$ denotes the Sobolev space endowed with the usual norm $\|\cdot \|_{1,r}$. First, we introduce the Lebesgue and Sobolev spaces with variable exponents and some properties that will be useful in our treatment. For a detailed overview we refer to the book of Diening-Harjulehto-H\"{a}st\"{o}-R$\mathring{\text{u}}$\v{z}i\v{c}ka \cite{Diening-Harjulehto-Hasto-Ruzicka-2011}. For any $r \in C(\close)$, we set
\begin{align*}
	r_+:=\max\limits_{x \in \close} r(x)
	\quad \text{and} \quad
	r_-:=\min\limits_{x \in \close} r(x),
\end{align*}
and define
\begin{align*}
	C_+(\close)=\{ r \in C(\close)\,:\, r_->1\}.
\end{align*}
Denoting by $M(\Omega)$ the space of all measurable functions $u\colon \Omega\to\R$, we define for any $r \in C_+(\close)$ the Lebesgue space with variable exponent by
\begin{align*}
	\Lp{r(\cdot)} = \{ u \in M(\Omega) \, : \,\rho_{r(\cdot)}(u)< \infty \},
\end{align*}
where the modular is given by
\begin{align*}
	\rho_{r(\cdot)}(u) = \into |u|^{r(x)} \,\diff x,
\end{align*}
endowed with the Luxemburg norm
\begin{align*}
	\norm{u}_{r(\cdot)} = \inf \l\{ \tau >0 \,:\, \rho_{r(\cdot)}\l(\frac{u}{\tau}\r)\le1 \r\}.
\end{align*}

Here, we recall the relation between the norm and the modular, see Fan-Zhao \cite[Theorems 1.2 and 1.3]{Fan-Zhao-2001}.

\begin{proposition}
	\label{properties_norm_modular_r}
	Let $r \in C_+(\close)$, $u \in \Lp{r(\cdot)}$ and $\lambda\in\R$. Then the following hold:
	\begin{enumerate}
		\item[\textnormal{(i)}]
			If $u\neq 0$, then $\|u\|_{r(\cdot)}=\lambda \quad \Longleftrightarrow \quad \rho_{r(\cdot)}(\frac{u}{\lambda})=1$;
		\item[\textnormal{(ii)}]
			$\|u\|_{r(\cdot)}<1$ (resp.\,$>1$, $=1$) $\quad \Longleftrightarrow \quad \rho_{r(\cdot)}(u)<1$ (resp.\,$>1$, $=1$);
		\item[\textnormal{(iii)}]
			If $\|u\|_{r(\cdot)}<1 \quad \Longrightarrow \quad \|u\|_{r(\cdot)}^{r_+}\leq \rho_{r(\cdot)}(u)\leq\|u\|_{r(\cdot)}^{r-}$;
		\item[\textnormal{(iv)}]
			If $\|u\|_{r(\cdot)}>1 \quad \Longrightarrow \quad \|u\|_{r(\cdot)}^{r_-}\leq \rho_{r(\cdot)}(u)\leq\|u\|_{r(\cdot)}^{r_+}$;
		\item[\textnormal{(v)}]
			$\|u\|_{r(\cdot)}\to 0 \quad \Longleftrightarrow \quad  \rho_{r(\cdot)}(u)\to 0$;
		\item[\textnormal{(vi)}]
			$\|u\|_{r(\cdot)}\to 1 \quad \Longleftrightarrow \quad  \rho_{r(\cdot)}(u)\to 1$;
		\item[\textnormal{(vii)}]
			$\|u\|_{r(\cdot)} \to +\infty \quad \Longleftrightarrow \quad  \rho_{r(\cdot)}(u)\to +\infty$;
		\item[\textnormal{(viii)}]
			$u_n \to u$ in $\Lp{r(\cdot)} \quad \Longrightarrow \quad \rho_{r(\cdot)}(u_n)\to \rho(u)$.
	\end{enumerate}
\end{proposition}
For $r' \in C_+(\close)$ being the conjugate variable exponent to $r$, that is,
\begin{align*}
	\frac{1}{r(x)}+\frac{1}{r'(x)}=1 \quad\text{for all }x\in\close,
\end{align*}
it is clear that $\Lp{r(\cdot)}^*=\Lp{r'(\cdot)}$ and the following H\"older's inequality hold
\begin{align*}
	\|uv\|_1\leq 2 \|u\|_{r(\cdot)}\|v\|_{r'(\cdot)},
\end{align*}
for all $u\in \Lp{r(\cdot)}$ and for all $v \in \Lp{r'(\cdot)}$, see Diening-Harjulehto-H\"{a}st\"{o}-R$\mathring{\text{u}}$\v{z}i\v{c}ka \cite[Lemma 3.2.20]{Diening-Harjulehto-Hasto-Ruzicka-2011}.

Furthermore, for $r_1, r_2\in C_+(\close)$ with $r_1(x) \leq r_2(x)$ for all $x\in \close$, we have the continuous embedding
\begin{align*}
	\Lp{r_2(\cdot)} \hookrightarrow \Lp{r_1(\cdot)}.
\end{align*}

Moreover, we can define variable exponent Lebesgue spaces with weights: for any $\omega \in \Lp{1}$, $\omega \geq 0$, we can define the modular
\begin{align*}
	\rho_{r(\cdot),\omega} (u) = \into \omega(x) |u|^{r(x)}\,\diff x.
\end{align*}
Then, we define the space
\begin{align*}
	L^{r(\cdot)}_\omega(\Omega)=\l\{u \in M(\Omega)\,:\, \into \rho_{r(\cdot),\omega} (u)\, \diff x<\infty \r\},
\end{align*}
endowed with the corresponding Luxemburg norm
\begin{align*}
	\|u\|_{r(\cdot),\omega} =\inf \l \{\lambda>0 \, : \, \rho_{r(\cdot),\omega} \left( \frac{u}{\lambda} \right)  \leq 1 \r\}.
\end{align*}

Next we can define the corresponding variable exponent Sobolev space $\Wp{r(\cdot)}$ which is given by, for $r \in C_+(\close)$,
\begin{align*}
	\Wp{r(\cdot)}=\l\{ u \in \Lp{r(\cdot)} \,:\, |\nabla u| \in \Lp{r(\cdot)}\r\},
\end{align*}
equipped with the norm
\begin{align*}
	\|u\|_{1,r(\cdot)}=\|u\|_{r(\cdot)}+\|\nabla u\|_{r(\cdot)},
\end{align*}
where $\|\nabla u\|_{r(\cdot)}= \|\,|\nabla u|\,\|_{r(\cdot)}$. It is well known that $\Lp{r(\cdot)}$ and $\Wp{r(\cdot)}$ are separable and reflexive Banach spaces and possess an equivalent, uniformly convex norm, see for example Diening-Harjulehto-H\"ast\"o-R$\mathring{\text{u}}$\v{z}i\v{c}ka \cite{Diening-Harjulehto-Hasto-Ruzicka-2011}.

For any $r \in C_+(\close)$ with $r_+ < N$, we denote by $r^*$ and $r_*$ the critical Sobolev exponents, defined for all $x \in \close$ as follows
\begin{align}\label{def_critical_exp}
	r^*(x) = \frac{Nr(x)}{N-r(x)}
	\quad\text{and}\quad
	r_*(x) = \frac{(N-1)r(x)}{N-r(x)}.
\end{align}
Furthermore, let $\sigma$ be the $(N-1)$-dimensional Hausdorff measure on the boundary $\partial\Omega$ and indicate with $L^{r(\cdot)}(\partial\Omega)$ the boundary Lebesgue space endowed with the usual norm $\norm{\cdot}_{r,\partial\Omega}$. We can consider a trace operator, i.e., a continuous linear operator $\gamma\colon\Wp{r(\cdot)} \to L^{m(\cdot)}(\partial\Omega)$ for all $m \in C(\close)$ with $1 \le m(x) < r_*(x)$ for every $x \in \close$, such that
\begin{align*}
	\gamma(u) = u \vert_{\partial\Omega} \quad \text{for all } u \in \Wp{r(\cdot)} \cap C(\close).
\end{align*}
If it also holds that $r \in W^{1,\gamma} (\Omega)$ with $\gamma > N$, then we can take any $m \in C(\close)$ with $1 \le m(x) \leq r_*(x)$ for every $x \in \close$.
By the trace embedding theorem, it is known that $\gamma$ is compact for any $r \in C(\close)$ with $1 \le r(x) < r_*(x)$ for all $x \in \close$, see Fan \cite[Corollary 2.4]{Fan-2008}. In this paper we avoid the notation of the trace operator and we consider all the restrictions of Sobolev functions to the boundary $\partial\Omega$ in the sense of traces. Moreover, we indicate with $\rho_{r(\cdot),\partial\Omega}(\cdot)$ and $\norm{\cdot}_{r(\cdot),\partial\Omega}$ the modular and the norm, respectively, of the space $L^{r(\cdot)}(\partial\Omega)$ with exponent $r(\cdot)$ on the boundary $\partial\Omega$.

By $C^{0, \frac{1}{|\log t|}}(\close)$ we denote the set of all functions $h\colon \close \to \R$ that are log-H\"older continuous, that is, there exists a constant $C>0$ such that
\begin{align*}
	|h(x)-h(y)| \leq \frac{C}{|\log |x-y||}\quad\text{for all } x,y\in \close \text{ with } |x-y|<\frac{1}{2}.
\end{align*}
Next, we present some embedding results, see Diening-Harjulehto-H\"{a}st\"{o}-R$\mathring{\text{u}}$\v{z}i\v{c}ka \cite[Corollary 8.3.2]{Diening-Harjulehto-Hasto-Ruzicka-2011}, Fan \cite[Corollary 2.4]{Fan-2008}, Fan \cite[Propositions 2.1 and 2.2]{Fan-2010}, Fan-Shen-Zhao \cite{Fan-Shen-Zhao-2001} and Ho-Kim-Winkert-Zhang \cite[Proposition 2.5]{Ho-Kim-Winkert-Zhang-2022}.

\begin{proposition}
	~
	\begin{enumerate}
		\item[\textnormal{(i)}]
			Let $r\in C^{0, \frac{1}{|\log t|}}(\close) \cap C_+(\close)$ and let $s\in C(\close)$ be such that $1\leq  s(x)\leq r^*(x)$ for all $x\in\close$. Then, the embedding $W^{1,r(\cdot)}(\Omega) \hookrightarrow L^{s(\cdot) }(\Omega)$ is  continuous. If $r\in C_+(\close)$, $s\in C(\close)$ and $1\leq s(x)< r^*(x)$ for all
			$x\in\overline{\Omega}$, then the embedding above is compact.
		\item[\textnormal{(ii)}]
			Suppose that $r\in C_+(\close)\cap W^{1,\gamma}(\Omega)$ for some $\gamma>N$ and let $s\in C(\close)$ be such that $1\leq  s(x)\leq r_*(x)$ for all $x\in\close$. Then, the embedding $W^{1,r(\cdot)}(\Omega)\hookrightarrow L^{s(\cdot) }(\partial \Omega)$ is continuous. If $r\in C_+(\close)$, $s\in C(\close)$ and $1\leq s(x)< r_*(x)$ for all $x\in\overline{\Omega}$, then the embedding above is compact.
	\end{enumerate}
\end{proposition}

\begin{remark}
	\label{remark-spaces}
	Note that for a bounded domain $\Omega\subset \R^N$ and $\gamma>N$ we have the following inclusions
	\begin{align*}
		C^{0,1}(\close)\subset \Wp{\gamma}\subset C^{0,1-\frac{N}{\gamma}}(\close) \subset C^{0, \frac{1}{|\log t|}}(\close).
	\end{align*}
\end{remark}

Now, we introduce the Musielak-Orlicz space, the Musielak-Orlicz Sobolev space and we  recall some properties that will be useful in the sequel. From now on, we assume the following:
\begin{enumerate}[label=\textnormal{(H)},ref=\textnormal{H}]
	\item\label{H}
	$p,q \in C(\close)$ such that $1<p(x)<N$ and $p(x)<q(x)<p^*(x)$ for all $x \in \close$ and $\mu\in L^\infty(\Omega)$ with $\mu(x)\geq0$ for a.a.\,$x\in\Omega$.
\end{enumerate}
We consider the nonlinear function $\mathcal{H}\colon \Omega \times [0,\infty)\to [0,\infty)$ defined by
\begin{align*}
	\mathcal H(x,t)= t^{p(x)}+\mu(x)t^{q(x)} \quad \text{for all } (x,t) \in \Omega \times [0,\infty),
\end{align*}
and we denote by $\rho_{\mathcal{H}}(\cdot)$ the corresponding modular, namely
\begin{align*}
	\rho_{\mathcal{H}}(u) = \into \mathcal{H}(x,|u|)\,\diff x
	= \into \left( |u|^{p(x)}+\mu(x)|u|^{q(x)}\right)\, \diff x.
\end{align*}
Then, we indicate with $L^\mathcal{H}(\Omega)$ the Musielak-Orlicz space, given by
\begin{align*}
	L^\mathcal{H}(\Omega)=\l\{u\in M(\Omega)\, : \, \rho_{\mathcal{H}}(u)<+\infty \r\},
\end{align*}
endowed with the Luxemburg norm
\begin{align*}
	\|u\|_{\mathcal{H}} = \inf \l\{ \tau >0 \,:\, \rho_{\mathcal{H}}\l(\frac{u}{\tau}\r) \leq 1  \r\}.
\end{align*}
Let $\WH$ be the Musielak-Orlicz Sobolev space, defined by
\begin{align*}
	\WH= \l\{u \in \Lp{\mathcal{H}} \,:\, |\nabla u| \in \Lp{\mathcal{H}} \r\},
\end{align*}
equipped with the usual norm
\begin{align*}
	\|u\|_{1,\mathcal{H}}= \|\nabla u \|_{\mathcal{H}}+\|u\|_{\mathcal{H}},
\end{align*}
where $\|\nabla u\|_\mathcal{H}=\|\,|\nabla u|\,\|_{\mathcal{H}}$. From Crespo-Blanco-Gasi\'nski-Harjulehto-Winkert \cite[Proposition 2.12]{Crespo-Blanco-Gasinski-Harjulehto-Winkert-2022} we know that $\Lp{\mathcal{H}}$ and $\Wp{\mathcal{H}}$ are reflexive Banach spaces. Further, we introduce the seminormed space
\begin{align*}
	L^{q(\cdot)}_\mu(\Omega)=\left \{u\in M(\Omega) \,:\, \into \mu(x) | u|^{q(x)} \,\diff x< +\infty \right \},
\end{align*}
and endow it with the seminorm
\begin{align*}
	\|u\|_{q(\cdot),\mu} =\inf \left \{ \tau >0 \,:\, \into \mu(x) \l(\frac{|u|}{\tau}\r)^{q(x)} \,\diff x  \leq 1  \right \}.
\end{align*}
The following result about the main embeddings of $\WH$ can be found in Crespo-Blanco-Gasi\'nski-Harjulehto-Winkert \cite[Propositions 2.16 and 2.18]{Crespo-Blanco-Gasinski-Harjulehto-Winkert-2022}.

\begin{proposition}
	\label{proposition_embeddings}
	Let \eqref{H} be satisfied. Then the following embeddings hold:
	\begin{enumerate}
		\item[\textnormal{(i)}]
			$\Lp{\mathcal{H}} \hookrightarrow \Lp{r(\cdot)}$ and $\WH \hookrightarrow \Wp{r(\cdot)}$ are continuous for all $r \in C(\close)$ with $1 \le r(x) \le p(x)$ for all $x \in \close$;
		\item[\textnormal{(ii)}]
			if $p \in C_+(\close)\cap C^{0,\frac{1}{|\log t|}}(\close)$, then $\WH \hookrightarrow \Lp{r(\cdot)}$ is continuous for $r \in C(\close)$ with $1 \le r(x) \le p^*(x)$ for all $x \in \close$;
		\item[\textnormal{(iii)}]
			$\WH \hookrightarrow \Lp{r(\cdot)}$ is compact for all $r \in C(\close)$ with $1 \le r(x) < p^*(x)$ for all $x \in \close$.
		\item[\textnormal{(iv)}]
			if $p \in C_+(\close)\cap W^{1,\gamma}(\Omega)$ for some $\gamma > N$, then $\WH \hookrightarrow L^{r(\cdot)}(\partial \Omega)$ is continuous for $r \in C(\close)$ with $1 \le r(x) \le p_*(x)$ for all $x \in \close$;
		\item[\textnormal{(v)}]
			$\WH \hookrightarrow L^{r(\cdot)}(\partial \Omega)$ is compact for $r \in C(\close)$ with $1 \le r(x) < p_*(x)$ for all $x \in \close$;
		\item[\textnormal{(vi)}]
			$\Lp{\mathcal{H}} \hookrightarrow L^{q(\cdot)}_\mu(\Omega)$ is continuous;
		\item[\textnormal{(vii)}]
			$L^{q(\cdot)}(\Omega) \hookrightarrow \Lp{\mathcal{H}}$ is continuous;
		\item[\textnormal{(viii)}]
			$\Wp{\mathcal{H}}\hookrightarrow \Lp{\mathcal{H}}$ is compact.
	\end{enumerate}
\end{proposition}

For our existence results, we equip the space $\WH$ with the following norm
\begin{equation}
	\begin{aligned}
		\label{norm}
		\norm{u} = \inf \Bigg\{ \tau >0 \, \colon
		& \into \l( \l| \frac{\nabla u}{\tau}\r|^{p(x)} + \mu(x) \l| \frac{\nabla u}{\tau}\r|^{q(x)} \r) \,\diff x \\
		& + \into \l| \frac{u}{\tau}\r|^{p(x)}\, \diff x +\intor \l| \frac{u}{\tau}\r|^{p(x)}\, \diff \sigma  \le 1\Bigg\},
	\end{aligned}
\end{equation}
induced by the modular
\begin{align*}
	\rho(u) = \into \left( |\nabla u|^{p(x)}+\mu(x)|\nabla u|^{q(x)}\right) \,\diff x + \into |u|^{p(x)} \,\diff x + \intor |u|^{p(x)} \,\diff \sigma,
\end{align*}
for all $u \in \WH$. We emphasize that in Section \ref{Section:equivalent norm} we prove in Proposition \ref{proposition_equivalent_norm} the existence of a new equivalent norm in $\WH$, denoted by $\|\cdot\|_{1,\mathcal{H}}^*$, in a more general setting and the  norm \eqref{norm} derives from $\|\cdot\|_{1,\mathcal{H}}^*$ defined in \eqref{norm_general} by choosing $\vartheta_1 \equiv \vartheta_2 \equiv 1$ and $\delta_1 \equiv \delta_2 \equiv p$. For reader's convenience, we give here the relationship between the modular $\rho(\cdot)$ and the norm $\|\cdot\|$, while in Section \ref{Section:equivalent norm} we present the same proposition for the general ones, see Proposition \ref{properties_modular_norm_general}.
\begin{proposition}\label{properties_modular_norm_complete}
	Let hypothesis \eqref{H} be satisfied, $u\in \WH$ and $\lambda\in\R$. Then the following hold:
	\begin{enumerate}
		\item[\textnormal{(i)}]
			If $u\neq 0$, then $\|u\|=\lambda \quad \Longleftrightarrow \quad \rho(\frac{u}{\lambda})=1$;
		\item[\textnormal{(ii)}]
			$\|u\|<1$ (resp.\,$>1$, $=1$) $\quad \Longleftrightarrow \quad \rho(u)<1$ (resp.\,$>1$, $=1$);
		\item[\textnormal{(iii)}]
			If $\|u\|<1 \quad \Longrightarrow \quad \|u\|^{q_+}\leq \rho(u)\leq\|u\|^{p_-}$;
		\item[\textnormal{(iv)}]
			If $\|u\|>1 \quad \Longrightarrow \quad \|u\|^{p_-}\leq \rho(u)\leq\|u\|^{q_+}$;
		\item[\textnormal{(v)}]
			$\|u\|\to 0 \quad \Longleftrightarrow \quad  \rho(u)\to 0$;
		\item[\textnormal{(vi)}]
			$\|u\|\to +\infty \quad \Longleftrightarrow \quad  \rho(u)\to +\infty$;
		\item[\textnormal{(vii)}]
			$\|u\|\to 1 \quad \Longleftrightarrow \quad  \rho(u)\to 1$;
	\end{enumerate}
\end{proposition}

Moreover, for any $h \in \R$ let $h^+=\max\{h,0\}$ and $h^-=\max\{-h,0\}$, then one has that $h=h^+-h^-$ and $|h|=h^++h^-$. Also, from Crespo-Blanco-Gasi\'nski-Harjulehto-Winkert \cite[Proposition 2.17]{Crespo-Blanco-Gasinski-Harjulehto-Winkert-2022} we know that, under assumption \eqref{H}, if $u \in \WH$ then $u^\pm \in \WH$.

Now, denote by $\lan\,\cdot\,,\,\cdot\,\ran$ the duality pairing between $\WH$ and its dual space $\WH^*$ and by $A\colon \WH \to \WH^*$ the nonlinear operator defined for all $u, v \in \WH$ by
\begin{align*}
	\lan A(u), v\ran =
	& \into \l( |\nabla u|^{p(x)-2}\nabla u+ \mu(x)|\nabla u|^{q(x)-2}\nabla u\r)\cdot \nabla v  \, \diff x \\
	& + \into |u|^{p(x)-2} u v \, \diff x + \intor |u|^{p(x)-2} u v \, \diff \sigma.
\end{align*}
In the following proposition we give the properties of this operator, see Proposition \ref{properties_general_operator_double_phase} in Section \ref{Section:equivalent norm}.

\begin{proposition}
	\label{properties_operator_double_phase}
	Let hypothesis \eqref{H} be satisfied. Then, the operator $A\colon \WH$ $\to \WH^*$ is bounded, continuous, strictly monotone and of type $(\Ss_+)$, that is,
	\begin{align*}
		\text{if} \quad	u_n\weak u \quad \text{in }\WH \quad\text{and}\quad  \limsup_{n\to\infty}\,\lan A(u_n),u_n-u\ran\le 0,
	\end{align*}
	then $u_n\to u$ in $\WH$. Moreover, it is
	coercive and a homeomorphism.
\end{proposition}

Next, we recall some tools needed in our investigations. In the sequel, for $X$ being a Banach space, we denote by $X^*$ its topological dual space.

\begin{definition}\label{cerami_condition}
	Given $L\in C^1(X)$, we say that $L$ satisfies the Cerami condition (\textnormal{C}-condition for short), if every sequence $\{u_n\}_{n\in\N}\subseteq X$ such that
	\begin{enumerate}[itemsep=0.2cm,label=\textnormal{(C${\arabic*}$)},ref=\textnormal{C${\arabic*}$}]
		\item\label{C1}
			$\{L(u_n)\}_{n\geq 1} \subseteq \R$ is bounded,
		\item\label{C2}
			$\l(1+\|u_n\|_X\r)L'(u_n)\to 0$ in $X^*$ as $n\to \infty$,
	\end{enumerate}
	admits a strongly convergent subsequence in $X$. We say that $L$ satisfies the Cerami condition at level $c \in \R$ (\textnormal{C$_c$}-condition for short), if \eqref{C1} is replaced by $L(u_n) \to c$ as $n \to \infty$.
\end{definition}

The following version of the Mountain-Pass Theorem is stated in the book of Papageorgiou-R\u{a}dulescu-Repov\v{s} \cite[Theorem 5.4.6]{Papageorgiou-Radulescu-Repovs-2019}.

\begin{theorem}\label{mountain_pass_theorem}
	Let $X$ be a Banach space and suppose $\varphi \in C^1(X), u_0, u_1 \in X$ with $\norm{u_1-u_0}>\delta>0$,
	\begin{align*}
		& \max\l\{\varphi(u_0), \varphi(u_1)\r\} \le \inf \l\{\varphi(u) \,:\, \norm{u-u_0}=\delta \r\} = m_\delta,                                                                   \\
		& c=\inf_{\gamma \in \Gamma} \max_{0 \le t \le 1} \varphi(\gamma(t)) \quad \text{with} \quad
		\Gamma = \l\{ \gamma \in C\l([0,1],X\r) \,:\, \gamma(0)=u_0, \gamma(1)=u_1\r\},
	\end{align*}
	and $\varphi$ satisfies the \textnormal{C$_c$}-condition. Then $c \ge m_\delta$ and $c$ is a critical value of $\varphi$. Moreover, if $c=m_\delta$, then there exists $u \in B_\delta(u_0)$ such that $\varphi'(u)=0$.
\end{theorem}

Finally, we present a version of the Quantitative Deformation Lemma, which can be found in Willem \cite[Lemma 2.3]{Willem-1996}.
\begin{lemma}
	\label{quantitative_deformation_lemma}
	Let $X$ be a Banach space, $\varphi \in C^1(X;\R)$, $\emptyset \ne S \subset X$, $c \in \R$, $\varepsilon, \delta>0$ such that
	\begin{align*}
		\norm{\varphi'(u)}_* \ge \frac{8\varepsilon}{\delta} \qquad \text{for all $u \in \varphi^{-1}\l([c-2\varepsilon, c+2\varepsilon]\r) \cap S_{2\delta}$},
	\end{align*}
	where $S_r = \{u \in X \,:\, d(u,S) = \inf_{u_0 \in S} \norm{u-u_0}<r\}$ for any $r>0$. Then there exists $\eta \in C([0,1]\times X;X)$ such that
	\begin{enumerate}
		\item[\textnormal{(i)}]
			$\eta(t,u)=u$, if $t=0$ or if $u \notin \varphi^{-1}\l([c-2\varepsilon, c+2\varepsilon]\r) \cap S_{2\delta}$,
		\item[\textnormal{(ii)}]
			$\varphi(\eta(1,u))\le c-\varepsilon$ for all $u \in \varphi^{-1}((-\infty,c+\varepsilon])\cap S$,
		\item[\textnormal{(iii)}]
			$\eta(t,\cdot)$ is an homeomorphism of $X$ for all $t \in [0,1]$,
		\item[\textnormal{(iv)}]
			$\norm{\eta(t,u)-u}\le \delta$ for all $u \in X$ and $t \in [0,1]$,
		\item[\textnormal{(v)}]
			$\varphi(\eta(\cdot,u))$ is decreasing for all $u \in X$,
		\item[\textnormal{(vi)}]
			$\varphi(\eta(t,u))<c$ for all $u \in \varphi^{-1}((-\infty,c]) \cap S_\delta$ and $t \in (0,1]$.
	\end{enumerate}
\end{lemma}

\section{A new equivalent norm}\label{Section:equivalent norm}
In this section we prove the existence of a new and general equivalent norm in $\Wp{\mathcal{H}}$. First, in addition to \eqref{H}, we suppose the following conditions:
\begin{enumerate}[label=\textnormal{(H$1$)},ref=\textnormal{H$1$}]
	\item\label{H1}
	\begin{enumerate}[label=\textnormal{(\roman*)},ref=\textnormal{\roman*}]
		\item
			$\delta_1, \delta_2\in C(\close)$ with $1\leq \delta_1(x)\leq p^*(x)$ and $1\leq \delta_2(x) \leq p_*(x)$ for all $x\in\close$, where
			\begin{enumerate}
				\item[\textnormal{(a$_1$)}]
				$p\in C(\close) \cap C^{0, \frac{1}{|\log t|}}(\close)$, if $\delta_1(x)=p^*(x)$ for some $x\in\close$;
				\item[\textnormal{(a$_2$)}]
				$p \in C(\close) \cap W^{1,\gamma}(\Omega)$ for some $\gamma>N$, if $\delta_2(x)=p_*(x)$ for some $x\in\close$;
			\end{enumerate}
		\item
			$\vartheta_1 \in \Linf$ with $\vartheta_1(x) \geq 0$ for a.a.\,$x\in\Omega$;
		\item
			$\vartheta_2 \in L^\infty(\partial\Omega)$ with $\vartheta_2(x)\geq 0$ for a.a.\,$x\in\partial\Omega$;
		\item\label{H1iv}
			$\vartheta_1 \not\equiv 0$ or $\vartheta_2\not\equiv 0$.
	\end{enumerate}
\end{enumerate}
In the sequel we use the seminormed spaces
\begin{align*}
	L^{\delta_1(\cdot)}_{\vartheta_1}(\Omega)
	& =\left \{u\in M(\Omega) \,:\, \into \vartheta_1(x) | u|^{\delta_1(x)}\, \diff x< \infty  \right \},\\
	L^{\delta_2(\cdot)}_{\vartheta_2}(\partial\Omega)
	& =\left \{u\in M(\Omega) \,:\, \intor \vartheta_2(x) | u|^{\delta_2(x)}\, \diff \sigma< \infty  \right \},
\end{align*}
with corresponding seminorms
\begin{align*}
	\|u\|_{\delta_1(\cdot),\vartheta_1}
	& =\inf \left \{ \tau >0 \,:\, \into \vartheta_1(x) \l|\frac{u}{\tau}\r|^{\delta_1(x)} \,\diff x  \leq 1  \right \},\\
	\|u\|_{\delta_2(\cdot),\vartheta_2,\partial\Omega}
	& =\inf \left \{ \tau >0 \,:\, \intor \vartheta_2(x) \l|\frac{u}{\tau}\r|^{\delta_2(x)} \,\diff \sigma  \leq 1  \right \},
\end{align*}
respectively. We set
\begin{align} \label{norm1}
	\|u\|_{1,\mathcal{H}}^{\circ}
	& = \|\nabla u\|_\mathcal{H}+\|u\|_{\delta_1(\cdot),\vartheta_1} +\|u\|_{\delta_2(\cdot),\vartheta_2,\partial\Omega},
\end{align}
and
\begin{align}\label{norm_general}
	\begin{split}
		\|u\|_{1,\mathcal{H}}^{*}
		& = \inf\Bigg\{\tau>0\,:\, \into \l(\l|\frac{\nabla u}{\tau}\r|^{p(x)}+\mu(x)\l|\frac{\nabla u}{\tau}\r|^{q(x)}\r)\,\diff x\\
		& \qquad \qquad \qquad +\into \vartheta_1(x) \l|\frac{u}{\tau}\r|^{\delta_1(x)}\,\diff x +\intor\vartheta_2(x)\l|\frac{u}{\tau}\r|^{\delta_2(x)}\,\diff \sigma\leq 1   \Bigg\}.
	\end{split}
\end{align}
It can be easily seen that $\|\cdot\|_{1,\mathcal{H}}^{\circ}$ and $\|\cdot\|_{1,\mathcal{H}}^{*}$ are norms on $\WH$. In the next result, we prove that they are both equivalent to the usual one.

\begin{proposition}\label{proposition_equivalent_norm}
	Let hypotheses \eqref{H} and \eqref{H1} be satisfied. Then, $\|\cdot\|_{1,\mathcal{H}}^{\circ}$ and $\|\cdot\|_{1,\mathcal{H}}^{*}$ given in \eqref{norm1} and \eqref{norm_general}, respectively, are both equivalent norms on $\WH$.
\end{proposition}

\begin{proof}
	We only prove the result when $\delta_1(x)=p^*(x)$ and $\delta_2(x)=p_*(x)$ for all $x\in\close$, the other cases can be shown in a similar way. So, we suppose that $p \in C(\close) \cap W^{1,\gamma}(\Omega)$ for some $\gamma>N$. Then, by Remark \ref{remark-spaces} we know that $p\in C(\close) \cap C^{0, \frac{1}{|\log t|}}(\close)$ as well.

	First, for $u\in\WH\setminus\{0\}$ we have
	\begin{align*}
		\into \vartheta_1(x) \left(\frac{|u|}{\|u\|_{p^*(\cdot)}}\right)^{p^*(x)}\,\diff x\leq \|\vartheta_1\|_\infty \, \rho_{p^*(\cdot)}\l(\frac{u}{\|u\|_{p^*(\cdot)}}\r)
		=\|\vartheta_1\|_\infty.
	\end{align*}
	Hence,
	\begin{align*}
		\|u\|_{p^*(\cdot),\vartheta_1}\leq \, \|\vartheta_1\|_\infty \|u\|_{p^*(\cdot)}.
	\end{align*}
	In the same way, we show that
	\begin{align*}
		\|u\|_{p_*(\cdot),\vartheta_2,\partial\Omega}\leq \|\vartheta_2\|_{\infty,\partial\Omega} \, \|u\|_{p_*(\cdot),\partial\Omega}.
	\end{align*}
	Using these along with Proposition \ref{proposition_embeddings}\textnormal{(ii)}, \textnormal{(iv)}, we obtain
	\begin{align*}
		\|u\|_{1,\mathcal{H}}^{\circ}
		 & \leq \|\nabla u\|_\mathcal{H}+ C_1\|u\|_{p^*(\cdot)}+C_2\|u\|_{p_*(\cdot),\partial\Omega} \\
		 & \leq \|\nabla u\|_\mathcal{H}+	C_3 \|u\|_{1,\mathcal{H}}+C_4\|u\|_{1,\mathcal{H}} \\
		 & \leq C_5 \|u\|_{1,\mathcal{H}},
	\end{align*}
	for all $u \in \WH$, with positive constants $C_i$, $i=1,\ldots 5$.

	Next, we are going to prove that
	\begin{align}
		\label{embedding_1}
		\|u\|_{\mathcal{H}} \leq C_6 \|u\|_{1,\mathcal{H}}^{\circ},
	\end{align}
	for some $C_6>0$. We argue indirectly and assume that \eqref{embedding_1} does not hold. Then, we find a sequence $\{u_n\}_{n\in\N}\subset \WH$ such that
	\begin{align}
		\label{embedding_2}
		\|u_n\|_{\mathcal{H}} > n\|u_n\|_{1,\mathcal{H}}^{\circ}\quad\text{for all }n\in\N.
	\end{align}
	Let $y_n=\frac{u_n}{\|u_n\|_{\mathcal{H}}}$. Hence, $\|y_n\|_{\mathcal{H}}=1$ and from \eqref{embedding_2} we get
	\begin{align}
		\label{embedding_3}
		\frac{1}{n}>\|y_n\|_{1,\mathcal{H}}^{\circ}.
	\end{align}
	From $\|y_n\|_{\mathcal{H}}=1$ and \eqref{embedding_3}, we know that  $\{y_n\}_{n\in\N}\subset \WH$ is bounded. Therefore, using the embeddings in Proposition \ref{proposition_embeddings}\textnormal{(ii)}, \textnormal{(iv)} and up to a subsequence if necessary, we may assume that
	\begin{align}
		\label{embedding_4}
		y_n\weak y \quad\text{in }\WH\quad\text{and}\quad y_n\weak y \quad\text{ in }\Lp{p^*(\cdot)} \text{ and }\Lprand{p_*(\cdot)}.
	\end{align}
	Furthermore, from \eqref{embedding_4} and Proposition \ref{proposition_embeddings}\textnormal{(viii)}, we conclude that $y_n\to y$ in $\Lp{\mathcal{H}}$ and because of $\|y_n\|_{\mathcal{H}}=1$ we have $y\neq 0$. Passing to the limit in \eqref{embedding_3} as $n\to \infty$ and using \eqref{embedding_4} along with the weak lower semicontinuity of the norm $\|\nabla \cdot\|_{\mathcal{H}}$ and of the seminorms $\|\cdot\|_{p^*(\cdot),\vartheta_1}, \|\cdot \|_{p_*(\cdot),\vartheta_2,\partial\Omega}$ we obtain
	\begin{align}
		\label{embedding_5}
		0 \geq \|\nabla y\|_\mathcal{H}+ \|y\|_{p^*(\cdot),\vartheta_1}+\|y\|_{p_*(\cdot),\vartheta_2,\partial\Omega}.
	\end{align}
	Inequality \eqref{embedding_5} implies that $y\equiv \eta \neq 0$ is a constant and so we have a contradiction
	\begin{align*}
		0\geq |\eta|  \|1\|_{p^*(\cdot),\vartheta_1}+|\eta|\|1\|_{p_*(\cdot),\vartheta_2,\partial\Omega}>0,
	\end{align*}
	because of \eqref{H1}\eqref{H1iv}. Therefore \eqref{embedding_1} holds and we get
	\begin{align*}
		\norm{u}_{1,\mathcal{H}} \le C_7 \norm{u}_{1,\mathcal{H}}^{\circ},
	\end{align*}
	for some $C_7>0$.

	Next, we are going to show that $\|\cdot\|_{1,\mathcal{H}}^{\circ}$ and $\|\cdot\|_{1,\mathcal{H}}^{*}$ are equivalent norms in $\WH$. For $u\in\WH$, we obtain
	\begin{align*}
		& \into \l(\l(\frac{|\nabla u|}{\|u\|_{1,\mathcal{H}}^{\circ}}\r)^{p(x)}+\mu(x)\l(\frac{|\nabla u|}{\|u\|_{1,\mathcal{H}}^{\circ}}\r)^{q(x)}\r)\,\diff x\\
		&+\into \vartheta_1(x) \l(\frac{|u|}{\|u\|_{1,\mathcal{H}}^{\circ}}\r)^{p^*(x)}\,\diff x+\intor \vartheta_2(x)\l(\frac{|u|}{ \|u\|_{1,\mathcal{H}}^{\circ}}\r)^{p_*(x)}\,\diff \sigma\\
		& \leq \rho_{\mathcal{H}}\l(\frac{\nabla u}{\|\nabla u\|_{\mathcal{H}}}\r)+\into \vartheta_1(x)\l(\frac{|u|}{\|u\|_{p^*(\cdot),\theta_1}}\r)^{p^*(x)}\,\diff x\\
		&\quad +\intor \vartheta_2(x)\l(\frac{|u|}{\|u\|_{p_*(\cdot),\theta_2,\partial\Omega}}\r)^{p_*(x)}\,\diff \sigma \\
		& =3.
	\end{align*}
	Thus,
	\begin{align}
		\label{equiv_1}
		\|u\|_{1,\mathcal{H}}^{*}\leq 3 \|u\|_{1,\mathcal{H}}^{\circ}.
	\end{align}
	On the other hand, we have
	\begin{align}\label{embedding_30}
		\begin{split}
			& \into \l(\l(\frac{|\nabla u|}{\|u\|_{1,\mathcal{H}}^{*}}\r)^{p(x)}+\mu(x)\l(\frac{|\nabla u|}{\|u\|_{1,\mathcal{H}}^{*}}\r)^{q(x)}\r)\,\diff x\\
			&+\into \vartheta_1(x) \l(\frac{|u|}{\|u\|_{1,\mathcal{H}}^{*}}\r)^{p^*(x)}\,\diff x +\int_{\partial\Omega}\vartheta_2(x)\l(\frac{|u|}{\|u\|_{1,\mathcal{H}}^{*}}\r)^{p_*(x)}\,\diff \sigma                                                                                              \\
			& \leq \rho_{1,\mathcal{H}}^*\l(\frac{ u}{\|u\|_{1,\mathcal{H}}^{*}}\r),
		\end{split}
	\end{align}
	where $\rho_{1,\mathcal{H}}^*$ is the corresponding modular to $\|\cdot\|_{1,\mathcal{H}}^{*}$ given by
	\begin{align*}
		\rho_{1,\mathcal{H}}^*(u)
		&=\into \l(|\nabla u|^{p(x)}+\mu(x)|\nabla u|^{q(x)}\r)\,\diff x+\into \vartheta_1(x)|u|^{p^*(x)}\,\diff x\\
		&+\int_{\partial\Omega} \vartheta_2(x)|u|^{p_*(x)}\,\diff \sigma.
	\end{align*}
	Note that, for $u \in \Wp{\mathcal{H}}$, the function $\tau \mapsto \rho_{1,\mathcal{H}}^*(\tau u)$ is continuous, convex and even and it is strictly increasing when $\tau \in [0,\infty)$. So, by definition, we directly obtain
	\begin{align*}
		\|u\|_{1,\mathcal{H}}^{*}=\tau \quad\text{if and only if}\quad \rho_{1,\mathcal{H}}^*\l(\frac{u}{\tau}\r)=1.
	\end{align*}
	From this and \eqref{embedding_30} we conclude that
	\begin{align*}
		\|\nabla u\|_\mathcal{H} \leq \|u\|_{1,\mathcal{H}}^{*}, \qquad \|u\|_{p^*(\cdot),\vartheta_1} \leq \|u\|_{1,\mathcal{H}}^{*} \quad \text{ and } \quad \|u\|_{p_*(\cdot),\vartheta_2,\partial\Omega}\leq \|u\|_{1,\mathcal{H}}^{*}.
	\end{align*}
	Therefore,
	\begin{align}
		\label{equiv_2}
		\frac{1}{3} \|u\|_{1,\mathcal{H}}^{\circ}\leq \|u\|_{1,\mathcal{H}}^{*}.
	\end{align}
	From \eqref{equiv_1} and \eqref{equiv_2} the proof is complete.
\end{proof}
Let
\begin{align*}
	r_1 := \min \l\{ p_- , (\delta_1)_-, (\delta_2)_- \r\} \quad \text{ and } \quad  r_2 := \max \l\{ q_+ , (\delta_1)_+, (\delta_2)_+ \r\}.
\end{align*}
In the following proposition we give the relation between the norm $\|\cdot\|_{1,\mathcal{H}}^{*}$ and the related modular function $\rho_{1,\mathcal{H}}^*(\cdot)$. The proof is similar to that one of
Propositions 2.13 and 2.14  given by Crespo-Blanco-Gasi\'{n}ski-Harjulehto-Winkert
in \cite{Crespo-Blanco-Gasinski-Harjulehto-Winkert-2022}.
\begin{proposition}
	\label{properties_modular_norm_general}
	Let hypotheses \eqref{H} and \eqref{H1} be satisfied, $u\in \WH$ and $\lambda\in\R$. Then the following hold:
	\begin{enumerate}
		\item[\textnormal{(i)}]
			If $u\neq 0$, then $\|u\|_{1,\mathcal{H}}^*=\lambda \quad \Longleftrightarrow \quad \rho_{1,\mathcal{H}}^*(\frac{u}{\lambda})=1$;
		\item[\textnormal{(ii)}]
			$\|u\|_{1,\mathcal{H}}^*<1$ (resp.\,$>1$, $=1$) $\quad \Longleftrightarrow \quad \rho_{1,\mathcal{H}}^*(u)<1$ (resp.\,$>1$, $=1$);
		\item[\textnormal{(iii)}]
			If $\|u\|_{1,\mathcal{H}}^*<1 \quad \Longrightarrow \quad \l(\|u\|_{1,\mathcal{H}}^*\r)^{r_2} \leq \rho_{1,\mathcal{H}}^*(u) \leq \l(\|u\|_{1,\mathcal{H}}^*\r)^{r_1}$;
		\item[\textnormal{(iv)}]
			If $\|u\|_{1,\mathcal{H}}^*>1 \quad \Longrightarrow \quad \l(\|u\|_{1,\mathcal{H}}^*\r)^{r_1}\leq \rho_{1,\mathcal{H}}^*(u) \leq \l(\|u\|_{1,\mathcal{H}}^*\r)^{r_2}$;
		\item[\textnormal{(v)}]
			$\|u\|_{1,\mathcal{H}}^*\to 0 \quad \Longleftrightarrow \quad  \rho_{1,\mathcal{H}}^*(u)\to 0$;
		\item[\textnormal{(vi)}]
			$\|u\|_{1,\mathcal{H}}^*\to \infty  \quad \Longleftrightarrow \quad  \rho_{1,\mathcal{H}}^*(u)\to \infty $;
		\item[\textnormal{(vii)}]
			$\|u\|_{1,\mathcal{H}}^*\to 1 \quad \Longleftrightarrow \quad  \rho_{1,\mathcal{H}}^*(u)\to 1$.
	\end{enumerate}
\end{proposition}

Finally, denote by $B\colon \WH \to \WH^*$ the nonlinear operator defined pointwise by
\begin{align*}
	\lan B(u), v\ran =
	& \into \l( |\nabla u|^{p(x)-2}\nabla u+ \mu(x)|\nabla u|^{q(x)-2}\nabla u\r)\cdot \nabla v  \, \diff x \\
	& + \into \vartheta_1 |u|^{\delta_1(x)-2} u v \, \diff x + \intor \vartheta_2|u|^{\delta_2(x)-2} u v \, \diff \sigma,
\end{align*}
for all $u, v \in \WH$. Arguing as in the in proof of Propositions 3.4 and 3.5 in \cite{Crespo-Blanco-Gasinski-Harjulehto-Winkert-2022}, we  have the following the properties.
\begin{proposition}
	\label{properties_general_operator_double_phase}
	Let hypotheses \eqref{H} and \eqref{H1} be satisfied. Then, the operator $B\colon \WH$ $\to \WH^*$ is bounded, continuous and strictly monotone. If, in addition, $1< \delta_1(x), \delta_2(x)$ for all $x\in\close$, then $B$ is coercive, a homeomorphism and of type $(\Ss_+)$.
\end{proposition}

\begin{proof}
	As in the proof of Proposition \ref{proposition_equivalent_norm} we only consider the case when $\delta_1(x)=p^*(x)$ and $\delta_2(x)=p_*(x)$ for all $x\in\close$. Similarly to the proof of Theorem 3.3 in \cite{Crespo-Blanco-Gasinski-Harjulehto-Winkert-2022}, we can show that $B$ is bounded, continuous and strictly monotone. Let us only show the proof for the $(\Ss_+)$-property. To this end, let $\{u_n\}_{n \in \N} \subseteq W^{1,\mathcal{H}}(\Omega)$ be a sequence such that
	\begin{align}\label{assumption_splus}
		u_n \weak u
		\quad\text{in }\WH
		\quad \text{and}\quad
		\limsup_{n \to \infty}\,\left\langle B(u_n), u_n - u \right\rangle \leq 0.
	\end{align}
	From Proposition \ref{proposition_embeddings}(ii) and (iv) we know, up to a subsequence if necessary, that
	\begin{align}\label{weak-convergence}
		u_n\weak u \quad\text{in }L^{p^*(\cdot)}(\Omega)
		\quad\text{and}\quad
		u_n\weak u \quad\text{in }L^{p_*(\cdot)}(\partial\Omega).
	\end{align}
	The strict monotonicity of $B$ implies that
	\begin{align*}
		\lim_{n \to \infty} \left\langle B(u_n) - B(u) , u_n - u \right\rangle
		= 0=\lim_{n \to \infty} \left\langle B(u) , u_n - u \right\rangle.
	\end{align*}
	This yields
	\begin{align*}
		&\lim_{n \to \infty} \into
		\left( |\nabla u_n|^{p(x)-2} \nabla u_n - |\nabla u|^{p(x)-2} \nabla u \right) \cdot \left( \nabla u_n - \nabla u \right) \,\diff x = 0,\\
		&\lim_{n \to \infty} \into \vartheta_1(x)
		\left( |u_n|^{p^*(x)-2} u_n - |u|^{p^*(x)-2} u \right) (u_n - u) \,\diff x = 0,\\
		&\lim_{n \to \infty} \intor \vartheta_2(x)
		\left( |u_n|^{p_*(x)-2} u_n - |u|^{p_*(x)-2} u \right) (u_n - u) \,\diff \sigma = 0.
	\end{align*}
	Then, in the same way as the claim in \cite[Proof of Theorem 3.3, after (3.2)]{Crespo-Blanco-Gasinski-Harjulehto-Winkert-2022}, taking \eqref{weak-convergence} into account, we can show that
	\begin{align}\label{strong-convergence}
		\begin{split}
			\nabla u_n &\to \nabla u \quad\text{in } L^{p(\cdot)}(\Omega),\\
			u_n&\to u \quad\text{in }L^{p^*(\cdot)}_{\vartheta_1}(\Omega),\\
			u_n&\to u \quad\text{in }L^{p_*(\cdot)}_{\vartheta_2}(\partial\Omega).
		\end{split}
	\end{align}

	From \eqref{strong-convergence} we know that
	\begin{align}\label{strong-convergence2}
		\begin{split}
			\nabla u_n &\to \nabla u \quad\text{in measure in } \Omega,\\
			\vartheta_1(x)^{\frac{1}{p^*(x)}}u_n&\to \vartheta_1(x)^{\frac{1}{p^*(x)}}u \quad\text{in measure in }\Omega,\\
			\vartheta_2(x)^{\frac{1}{p_*(x)}}u_n&\to \vartheta_2(x)^{\frac{1}{p_*(x)}}u \quad\text{in measure in }\partial\Omega.
		\end{split}
	\end{align}

	Note that if $a_n, b_n\geq 0$ for all $n\in\N$, we have
	\begin{align}\label{limsup-property}
		\limsup_{n\to\infty} a_n \leq \limsup_{n\to\infty} (a_n+b_n).
	\end{align}
	Therefore, from \eqref{limsup-property}, the $\limsup$-condition in \eqref{assumption_splus} in the shape
	\begin{align*}
		\limsup_{n \to \infty} \left\langle B(u_n) - B(u) , u_n - u \right\rangle
		\leq 0
	\end{align*}
	and the weak convergence of \eqref{assumption_splus} as well as the embeddings $\WH\hookrightarrow L^{p^*(\cdot)}_{\vartheta_1}(\Omega)$, $\WH\hookrightarrow L^{p_*(\cdot)}_{\vartheta_2}(\partial\Omega)$, we obtain that
	\begin{align*}
		&\limsup_{n\to\infty} \into
		\left( |\nabla u_n|^{p(x)-2} \nabla u_n +\mu(x) |\nabla u_n|^{q(x)-2} \nabla u_n \right) \cdot (\nabla u_n - \nabla u)  \,\diff x \leq 0,\\
		&\limsup_{n \to \infty} \into \vartheta_1(x)
		|u_n|^{p^*(x)-2} u_n (u_n - u) \,\diff x \leq  0,\\
		&\limsup_{n \to \infty} \intor \vartheta_2(x)
		|u_n|^{p_*(x)-2} u_n (u_n - u) \,\diff \sigma \leq 0.
	\end{align*}
	Arguing as in \cite[(3.8), (3.9) and (3.10)]{Crespo-Blanco-Gasinski-Harjulehto-Winkert-2022} it can be shown that
	\begin{align}\label{proof_4}
		\begin{split}
			&\lim_{n \to \infty} \int_\Omega \l(\frac{|\nabla u_n|^{p(x)}}{p(x)} + \mu(x) \frac{|\nabla u_n|^{q(x)}}{q(x)}\r) \diff x\\
			&= \int_\Omega \l(\frac{|\nabla u|^{p(x)}}{p(x)} + \mu(x) \frac{|\nabla u|^{q(x)}}{q(x)}\r) \diff x,\\
			&\lim_{n \to \infty} \int_\Omega \vartheta_1(x)|u_n|^{p^*(x)}\,\diff x= \int_\Omega \vartheta_1(x)| u|^{p^*(x)}\,\diff x,\\
			&\lim_{n \to \infty} \intor \vartheta_2(x)|u_n|^{p_*(x)}\,\diff \sigma= \intor \vartheta_2(x)| u|^{p_*(x)}\,\diff \sigma.
		\end{split}
	\end{align}
	Due to \eqref{strong-convergence2}, the left-hand sides of \eqref{proof_4} converge in measure to those on the right-hand sides. Then, the converse of Vitali's theorem implies the uniform integrability of the sequences of functions
	\begin{align*}
		&\l\{\frac{|\nabla u_n|^{p(x)}}{p(x)}+ \mu(x) \frac{|\nabla u_n|^{q(x)}}{q(x)}\r\}_{n\in\N},\\
		&\l\{\vartheta_1(x)\frac{|u_n|^{p^*(x)}}{p^*(x)}\r\}_{n\in\N},\quad
		\l\{\vartheta_2(x)\frac{|u_n|^{p_*(x)}}{p_*(x)}\r\}_{n\in\N}.
	\end{align*}
	But then the sequences
	\begin{align*}
		&A_n:=\l\{|\nabla u_n-\nabla u|^{p(x)}+\mu(x)|\nabla u_n-\nabla u|^{q(x)}\r\}_{n\in\N},\\
		&B_n:=\l\{\vartheta_1(x)|u_n- u|^{p^*(x)}\r\}_{n\in\N},
		\quad C_n:=\l\{\vartheta_2(x)|u_n- u|^{p_*(x)}\r\}_{n\in\N},
	\end{align*}
	are uniformly integrable. This gives
	\begin{align*}
		0=\lim_{n \to \infty}\into A_n\,\diff x
		=\lim_{n \to \infty}\into B_n\,\diff x
		=\lim_{n \to \infty}\intor C_n\,\diff \sigma,
	\end{align*}
	which implies that
	\begin{align*}
		& \lim_{n \to \infty} \rho_{1,\mathcal{H}}^*( u_n - u )\\
		&=
		\lim_{n\to\infty} \left( \into \l(|\nabla u_n-\nabla u|^{p(x)}+\mu(x)|\nabla u_n-\nabla u|^{q(x)}\r)\,\diff x\right.\\
		&\left.\qquad \qquad+\into \vartheta_1(x)|u_n-u|^{p^*(x)}\,\diff x+\int_{\partial\Omega} \vartheta_2(x)|u_n-u|^{p_*(x)}\,\diff \sigma\right)=0.
	\end{align*}
	But this is equivalent to $\norm{u_n - u}_{1,\mathcal{H}} \to 0$, see Proposition \ref{properties_modular_norm_general} \textnormal{(v)}. Thus, $u_n\to u$ in $\WH$.
\end{proof}

\section{Bounded solutions}\label{Section:bounded solutions}
In this section we give a result about the boundedness in the $L^\infty$-norm of the solutions of \eqref{problem}. We state the theorem in a more general and more natural setting than in \eqref{problem} and even allow a gradient dependency on the nonlinearity at the  right-hand side in the domain. For this purpose we need the following assumptions.

\begin{enumerate}[label=\textnormal{(H$_\infty$)},ref=\textnormal{H$_\infty$}]
	\item\label{H_infty}
	Let $\mathcal{A} \colon \Omega \times \R \times \R^N \to \R^N$ and $\mathcal{B} \colon \Omega \times \R \times \R^N \to \R$ be Carath\'eodory functions and assume that there exist constants $a_1, a_2, a_3, b > 0$ and $r \in C_+(\close)$ with $q(x) < r(x) < p^*(x)$ for all $x \in \close$ such that
	\begin{align*}
		|\mathcal{A}(x,t,\xi)|
		& \leq a_1 \left[ |t|^{\frac{r(x)}{p'(x)}} + |\xi|^{p(x)-1} + \mu(x) |\xi|^{q(x)-1} + 1 \right],\\
		\mathcal{A}(x,t,\xi) \cdot \xi
		& \geq a_2 \left[ |\xi|^{p(x)} + \mu(x) |\xi|^{q(x)} \right] - a_3 \left[ |t|^{r(x)} + 1 \right], \\
		|\mathcal{B}(x,t,\xi)|
		& \leq b \left[ |\xi|^{\frac{p(x)}{r'(x)}} + |t|^{r(x) - 1} + 1 \right],
	\end{align*}
	for a.a.\,$x \in \Omega$ and for all $(t,\xi) \in \R \times \R^N$. Furthermore, let $\mathcal{C} \colon \partial \Omega \times \R \to \R$ be also a Carath\'eodory function, $c>0$ and $l \in C_+(\close)$ with $p(x) < l(x) < p_*(x)$ for all $x \in \close$ such that
	\begin{align*}
		| \mathcal{C} (x,t) |
		\leq c \left[ |t|^{l  (x) - 1} + 1 \right],
	\end{align*}
	for a.a.\,$x \in \partial\Omega$ and for all $t \in \R$.
\end{enumerate}

We consider the problem
\begin{equation}\label{problem-Linfty}
	\begin{aligned}
		-\operatorname{div} \mathcal{A}(x,u,\nabla u)
		 & = \mathcal{B}(x,u,\nabla u) \quad &  & \text{in } \Omega,\\
		\mathcal{A}(x,u,\nabla u) \cdot \nu
		& = \mathcal{C} (x,u) &  & \text{on } \partial\Omega,
	\end{aligned}
\end{equation}
already presented in the Introduction, see \eqref{pr}. We say that $u \in \WH$ is a weak solution of \eqref{problem-Linfty} if for all $v \in \WH$ it holds that
\begin{align*}
	\into \mathcal{A}(x,u,\nabla u) \cdot \nabla v \,\diff x
	= \into \mathcal{B}(x,u,\nabla u) v \,\diff x + \intor \mathcal{C} (x,u) v \,\diff \sigma.
\end{align*}

Following Theorem 4.3 due to Ho-Winkert \cite{Ho-Winkert-2022}, we obtain a priori $L^\infty$-estimates for the problem \eqref{problem-Linfty}.

\begin{theorem}
	\label{bounded-solutions}
	Let hypotheses \eqref{H} and \eqref{H_infty} be satisfied and let $u \in \WH$ be a weak solution of problem \eqref{problem-Linfty}. Then, $u \in \Lp{\infty}\cap L^\infty(\partial\Omega)$ and
	\begin{align*}
		\norm{u}_{\infty}+\norm{u}_{\infty,\partial\Omega} \leq C \max \left\lbrace \norm{u}_{r(\cdot)}^{\tau_1} , \norm{u}_{r(\cdot)}^{\tau_2}, \norm{u}_{l(\cdot),\partial \Omega}^{\tau_1} , \norm{u}_{l(\cdot),\partial \Omega}^{\tau_2} \right\rbrace,
	\end{align*}
	where $C, \tau_1,\tau_2 > 0$ are independent of $u$.
\end{theorem}

\begin{proof}
	We base our arguments on the proof of \cite[Theorem 4.3]{Ho-Winkert-2022} introducing the following changes.

	First, take
	\begin{align*}
		\Psi (x,t)& = t^{r(x)} \quad \text{for all } (x,t) \in \close \times [0,\infty), \\
		Z_n & = \int_{A_{\kappa_n}} (u - \kappa_n)^{r(x)} \,\diff x, \\
		\Upsilon (x,t) & = t^{l(x)} \quad \text{for all } (x,t) \in \close \times [0,\infty), \\
		Y_n & = \int_{\Gamma_{\kappa_n}} (u - \kappa_n)^{l(x)} \,\diff x,
	\end{align*}
	instead of the definitions given there. Then the Step 1 works exactly the same except for (4.9), which now is true because $q(x) \leq r(x)$ for all $x \in \close$ and
	\begin{align*}
		& \int_{A_{\kappa_{n+1}}} \left[ (u - \kappa_{n+1})^{p(x)} + \mu(x) (u - \kappa_{n+1})^{q(x)} \right]\, \diff x \\
		&\leq \int_{A_{\kappa_{n+1}}} \left[ (u - \kappa_{n+1})^{r(x)} + 1 + ||\mu||_{\infty}(u - \kappa_{n+1})^{r(x)} + ||\mu||_{\infty} \right] \,\diff x.
	\end{align*}
	Later, we take
	\begin{align*}
		T_{n,i} (\alpha) & = \int_{\Omega_i} v_n^\alpha \,\diff x \quad \text{for all } i \in \{1, \ldots , m\},\alpha > 0, \\
		H_{n,i} (\alpha) & = \int_{\Omega_i} v_n^\alpha\, \diff x \quad \text{for all } i \in I,\alpha > 0,
	\end{align*}
	skip the $\psi_\star$ and $\Phi_\star$ parts and then use only the embeddings
	\begin{equation}
		\label{embeddings-Linfty}
		\begin{aligned}
			 & W^{1,p(\cdot)} (\Omega_i) \hookrightarrow W^{1,(p_i)_-} (\Omega_i) \hookrightarrow L^{r_i^\star + \eps} (\Omega_i),\\
			 & W^{1,p(\cdot)} (\Omega_i) \hookrightarrow W^{1,(p_i)_-} (\Omega_i) \hookrightarrow L^{l_i^\star + \eps} (\partial \Omega_i).
		\end{aligned}
	\end{equation}
	Then one can complete Step 2 with a completely analogous argument. We finish the proof by repeating exactly the treatment of Step 3.
\end{proof}

\begin{remark}
Let us emphasize that Theorem \ref{bounded-solutions} holds under the weaker hypothesis on the exponents given in \eqref{H_infty} instead of the much more restrictive assumptions needed in \cite[Theorem 4.3]{Ho-Winkert-2022}. The reason behind this relies on the less general growth condition we
require on the main operators, so we only need to use the embeddings \eqref{embeddings-Linfty} instead of the other stronger and sharper embeddings used in (4.19) and (4.49) of \cite{Ho-Winkert-2022} and for which the authors require the aforementioned stronger hypothesis on the exponents.
\end{remark}

\section{Constant sign solutions}\label{Section:constant sign}
In this section we establish existence of two constant sign solutions obtained through Theorem \ref{mountain_pass_theorem}. In particular, one solution turns out to be nonnegative and the other one to be nonpositive. First, we have to strengthen the hypotheses \eqref{H} as follows:
\begin{enumerate}[label=\textnormal{(H2)},ref=\textnormal{H2}]
	\item\label{H2}
	$p,q \in C(\close)$ such that $1<p(x)<N$ and $p(x)<q(x)<(p_-)_*$ for all $x \in \close$ and $\mu\in L^\infty(\Omega)$ with $\mu(x)\geq0$ for a.a.\,$x\in\Omega$.
\end{enumerate}
Next, we state the required assumptions on the nonlinearities:
\begin{enumerate}[label=\textnormal{(H$_{f,g}$)},ref=\textnormal{H$_{f,g}$}]
	\item\label{H_{f,g}}
		Let $f\colon\Omega \times \R\to\R$ and $g\colon\partial\Omega \times \R\to\R$ be Carath\'eodory functions and $F(x,t)=\int_0^t f(x,s) \,\diff s$ and $G(x,t)=\int_0^t g(x,s) \,\diff s$ be such that the following hold:
		\begin{enumerate}[itemsep=0.2cm,label=\textnormal{(h$_{\arabic*}$)},ref=\textnormal{h$_{\arabic*}$}]
			\item\label{h_1}
				there exist $\ell, \kappa \in C_+(\close)$ and $K_1,K_2>0$ with $\ell_+<(p_-)^*$ and $\kappa_+<(p_-)_*$ such that
				\begin{align*}
					|f(x,t)| \leq K_1\l(1+|t|^{\ell(x) -1}\r) \quad
					& \text{for a.a.\,}x\in\Omega,\\
					|g(x,t)| \leq K_2\l(1+|t|^{\kappa(x) -1}\r) \quad
					& \text{for a.a.\,}x\in\partial\Omega,
				\end{align*}
				and for all $t\in\R$;
			\item\label{h_2}
				\begin{align*}
					\lim_{t\to \pm\infty} \frac{F(x,t)}{|t|^{q_+}}=\infty  \quad
					& \text{uniformly for a.a.\,}x\in\Omega,\\
					\lim_{t\to \pm\infty} \frac{G(x,t)}{|t|^{q_+}}=\infty  \quad
					& \text{uniformly for a.a.\,}x\in\partial\Omega;
				\end{align*}
			\item\label{h_3}
				\begin{align*}
					\lim_{t\to 0} \frac{F(x,t)}{|t|^{p(x)}}=0 \quad
					& \text{uniformly for a.a.\,}x\in\Omega,\\
					\lim_{t\to 0} \frac{G(x,t)}{|t|^{p(x)}}=0 \quad
					& \text{uniformly for a.a.\,}x\in\partial\Omega;
				\end{align*}
			\item\label{h_4}
				there exist $\alpha, \beta, \zeta, \theta \in C_+(\close)$ with
				\begin{align*}
					\min\{\alpha_-,\beta_-\}
					& \in \l((\ell_+-p_-)\frac{N}{p_-},\ell_+\r), \\
					\min\{\zeta_-,\theta_-\}
					& \in \l( (\kappa_+ - p_-) \frac{N-1}{p_- - 1} ,\kappa_+\r),
				\end{align*}
				and $K_3>0$ such that
				\begin{align*}
					0<K_3 \leq
					& \liminf_{t\to \infty }\, \frac{f(x,t)t-q_+F(x,t)}{|t|^{\alpha(x)}}, \\
					0<K_3 \leq
					& \liminf_{t\to -\infty}\, \frac{f(x,t)t-q_+F(x,t)}{|t|^{\beta(x)}},
				\end{align*}
				uniformly for a.a.\,$x\in\Omega$ and $K_4>0$ such that
				\begin{align*}
					0<K_4 \leq
					& \liminf_{t\to \infty }\, \frac{g(x,t)t-q_+G(x,t)}{|t|^{\zeta(x)}}, \\
					0<K_4 \leq
					& \liminf_{t\to -\infty}\, \frac{g(x,t)t-q_+G(x,t)}{|t|^{\theta(x)}},
				\end{align*}
				uniformly for a.a.\,$x\in\partial\Omega$;
			\item\label{h_5}
				the functions
				\begin{align*}
					t \mapsto \frac{f(x,t)}{|t|^{q_+-1}}
					\quad \text{and} \quad
					t \mapsto \frac{g(x,t)}{|t|^{q_+-1}}
				\end{align*}
				are increasing in $(-\infty, 0)$ and in $(0, \infty )$ for a.a.\,$x\in\Omega$ and for a.a.\,$x\in\partial\Omega$, respectively.
		\end{enumerate}
\end{enumerate}
We note that assumption \eqref{h_3} together with the continuity of $f(x,\cdot)$ and $g(x,\cdot)$ implies that
\begin{align}
	\label{function in zero}
	f(x,0)=0 \quad \text{for a.a.\,$x \in \Omega$ }
	\quad \text{and } \quad
	g(x,0)=0 \quad \text{for a.a.\,$x \in \partial\Omega$.}
\end{align}
Moreover, in Lemma 4.4 of Crespo-Blanco-Winkert \cite{Crespo-Blanco-Winkert-2022}, the authors summarize the properties that the nonlinear term of the equation (i.e. function $f$) verifies as consequences of the previous assumptions. Clearly, as the nonlinear function $g$ satisfies similar hypotheses on the boundary, it also verifies the same properties on $\partial\Omega$.

\begin{remark}
The conditions on the exponents in \eqref{h_4} are well defined since from \eqref{h_1} we have $\ell_+<(p_-)^*$ and $\kappa_+<(p_-)_*$ and the following hold
	\begin{align*}
		(\ell_+-p_-)\frac{N}{p_-}
		& = \ell_+\frac{N}{p_-} - (p_-)^*\frac{N-p_-}{p_-} < \ell_+\frac{N}{p_-} - \ell_+\frac{N-p_-}{p_-} = \ell_+, \\
		(\kappa_+ - p_-) \frac{N-1}{p_- - 1}
		& = \kappa_+\frac{N-1}{p_- -1} - (p_-)_*\frac{N-p_-}{p_- -1} <
		\kappa_+\frac{N-1}{p_- -1} - \kappa_+\frac{N-p_-}{p_- -1} = \kappa_+.
	\end{align*}
\end{remark}

\begin{example} \label{examples}
	Simple examples of $f$ and $g$ satisfying \eqref{H_{f,g}} are
	\begin{align*}
		f(x,t) = |t|^{q_+ + \eps_1 -2}t
		\quad\text{and}\quad
		g(x,t) = |t|^{q_+ + \eps_2 -2}t,
	\end{align*}
	i.e.\,independent of $x$, where $0 < \eps_1 < \min\{ (p_-)^* - q_+) , 1 \}$ and $0 < \eps_2 < \min\{ (p_-)_* - q_+) , 1\}$. For the assumption \eqref{h_4} choose $\alpha(x) = q_+ + \eps_1 - \tilde{\eps} -1$, where $\tilde{\eps}$ is small enough, and choose $\beta, \zeta$ and $\theta$ analogously.

	Less trivial examples of $f$ and $g$ are \begin{equation*}
		f(x,t) =
		\begin{cases}
			|t|^{l_1(x)-2}t[1 + \log(-t)], & \text{if } \phantom{-1}t \leq -1, \\
			|t|^{\eta(x)-2}t,                 & \text{if } -1 < t < 1,               \\
			|t|^{l_2(x)-2}t[1 + \log(t)],  & \text{if } \phantom{-1} 1 \leq t,
		\end{cases}
	\end{equation*}
	\begin{equation*}
		g(x,t) =
		\begin{cases}
			|t|^{\kappa_1(x)-2}t[1 + \log(-t)], & \text{if } \phantom{-1}t \leq -1, \\
			|t|^{\nu(x)-2}t,                 & \text{if } -1 < t < 1,               \\
			|t|^{\kappa_2(x)-2}t[1 + \log(t)],  & \text{if } \phantom{-1} 1 \leq t,
		\end{cases}
	\end{equation*}
	where $l_1, l_2, \eta \in C(\close)$, $q_+ \leq \eta(x)$ and $q_+ \leq l_1(x), l_2(x) < (p_-)^*$ for all $x \in \close$, and they satisfy
	\begin{equation*}
		\frac{ \max\{(l_1)_+,(l_2)_+\} }{p_-} - \frac{(l_i)_-}{N} < 1, \quad \text{for all } i \in \{ 1, 2\},
	\end{equation*}
	and also $\kappa_1, \kappa_2, \nu \in C(\close)$, $q_+ \leq \nu(x)$ and $q_+ \leq \kappa_1(x), \kappa_2(x) < (p_-)_*$ for all $x \in \close$, and they satisfy
	\begin{equation*}
		\frac{ \max\{(\kappa_1)_+,(\kappa_2)_+\} }{p_- - 1} - \frac{(\kappa_i)_-}{N-1} < \frac{p_-}{p_- - 1}, \quad \text{for all } i \in \{ 1, 2\}.
	\end{equation*}
	Then $f$ and $g$ satisfy all the assumptions above. For the assumption \eqref{h_2} of $f$ take $l(x) = \max \{ l_1(x), l_2(x) \} + \eps$ for all $x \in \close$, with $\eps > 0$ small enough so that $l_+ < (p_-)^*$ and
	\begin{equation*}
		\frac{ l_+ }{p_-} - \frac{(l_i)_-}{N} < 1, \quad \text{for all } i \in \{ 1, 2\}.
	\end{equation*}
	For the assumption \eqref{h_4} of $f$, take $\alpha(x) = l_1(x)$, $\beta(x) = l_2(x)$ for all $x \in \close$. This is the reason for the assumption on $(l_1)_\pm$ and $(l_2)_\pm$. Observe that if we take $l_1 = l_2 = l$ constant, the condition is equivalent to $l < (p_-)^*$, hence redundant in that case. For the assumptions \eqref{h_2} and \eqref{h_4} of $g$, analogous considerations apply.
\end{example}

Our aim is to establish results on the existence of weak solutions for problem \eqref{problem}, namely functions $u \in \WH$ such that
\begin{align*}
	&\into \l( |\nabla u|^{p(x)-2}\nabla u+ \mu(x) |\nabla u|^{q(x)-2}\nabla u \r)\cdot  \nabla v \,\diff x + \into |u|^{p(x)-2}u v \, \diff x \\
	&= \into f(x,u)v \, \diff x + \intor g(x,u)v \, \diff \sigma - \intor |u|^{p(x)-2}u v \, \diff \sigma,
\end{align*}
for every $v \in \WH$. In particular, these weak solutions are critical points of the energy functional $I\colon \WH \to \R$ associated to the problem \eqref{problem} given by
\begin{align*}
	I(u)
	&=  \into \l(\frac{|\nabla u|^{p(x)}}{p(x)}+\mu(x) \frac{|\nabla u|^{q(x)}}{q(x)}\r) \,\diff x + \into \frac{|u|^{p(x)}}{p(x)} \,\diff x \\
	& \quad + \intor \frac{|u|^{p(x)}}{p(x)} \,\diff \sigma - \into F(x,u) \,\diff x - \intor G(x,u) \,\diff \sigma ,
\end{align*}
for all $u \in \WH$. Since we are interested in constant sign solutions, we consider the positive and negative truncations of the functional $I$, that are $I_\pm \colon \WH \to \R$ defined by
\begin{align*}
	I_\pm(u)
	&= \into \l(\frac{|\nabla u|^{p(x)}}{p(x)}+\mu(x) \frac{|\nabla u|^{q(x)}}{q(x)}\r)\,\diff x + \into \frac{|u|^{p(x)}}{p(x)} \,\diff x \\
	& \quad + \intor \frac{|u|^{p(x)}}{p(x)} \,\diff \sigma - \into F (x,\pm u^\pm) \,\diff x - \intor G(x,\pm u^\pm) \,\diff \sigma ,
\end{align*}
for all $u \in \WH$, where we have taken \eqref{function in zero} into account. Our existence result is based on the Mountain-Pass Theorem. First we give preliminary results in order to verify the assumptions required in Theorem \ref{mountain_pass_theorem}. We start with the compactness condition on the functional.

\begin{proposition}
	\label{proposition_Cerami}
	Let hypotheses \eqref{H2}, \eqref{h_1}, \eqref{h_3} and \eqref{h_4} be satisfied. Then, the functionals $I_\pm$ satisfy the \textnormal{C}-condition.
\end{proposition}
\begin{proof}
	We show the proof for $I_+$, the case for $I_-$ works in the same way. Let $\{u_n\}_{n\in\N}\subseteq \WH$ be a sequence such that \eqref{C1} and \eqref{C2} from Definition \ref{cerami_condition} hold. From \eqref{C2}, there exists $\{\eps_n\}_{n \in \N}$ with $\eps_n\to0^+$ such that
	\begin{align}
		\label{cons_C_2}
		\l| \lan I'_+ (u_n),v \ran \r| \le \frac{\eps_n \norm{v}}{1+\norm{u_n}} \quad \text{for all } n \in \N \text{ and for all } v \in \WH.
	\end{align}
	Choosing $v=-u_n^- \in \WH$, one has
	\begin{align*}
		\rho(-u_n^-) - \into f(x,+u_n^+)(-u_n^-)\,\diff x - \intor g(x,+u_n^+)(-u_n^-)\,\diff\sigma \le \eps_n,
	\end{align*}
	for all $n \in \N$, which leads to $\rho(-u_n^-) \to 0$ as $n \to \infty$, since the supports of $+u_n^+$ and $-u_n^-$ do not overlap. From Proposition \ref{properties_modular_norm_complete}(v) it follows that
	\begin{align}
		\label{u_n_-_to_zero}
		- u_n^- \to 0 \quad \text{in $\WH$.}
	\end{align}

	{\bf Claim 1:} $\{u_n^+\}_{n \in \N}$ is bounded in $\Lp{\alpha_-}$ and in $L^{\zeta_-}(\partial\Omega)$.\\
	From \eqref{C1} we have that there exists a constant $M_1>0$ such that for all $n \in \N$ one has $|I_+(u_n)|\le M$, that is
	\begin{align*}
		\frac{1}{q_+} \rho(u_n^+) - \into F(x,u_n^+)\,\diff x - \intor G(x,u_n^+)\,\diff \sigma \le M_1 - \frac{1}{q_+} \rho(-u_n^-),
	\end{align*}
	which, taking \eqref{u_n_-_to_zero} into account, leads to
	\begin{align}\label{eq_a}
		\rho(u_n^+) - \into q_+ F(x,u_n^+)\,\diff x - \intor q_+ G(x,u_n^+)\,\diff \sigma \le M_2,
	\end{align}
	for all $n \in \N$ and for some $M_2>0$. Testing \eqref{cons_C_2} for $v=u_n^+$, we have
	\begin{align}\label{eq_b}
		-\rho(u_n^+) + \into f(x,u_n^+)u_n^+\,\diff x + \intor g(x,u_n^+)u_n^+\,\diff \sigma \le \eps_n,
	\end{align}
	for all $n \in \N$.	Adding \eqref{eq_a} and \eqref{eq_b} we obtain
	\begin{align}\label{eq_c}
		\begin{split}
			&\into \l( f(x,u_n^+)u_n^+ - q_+ F(x,u_n^+)\r) \,\diff x\\
			&+ \intor \l( g(x,u_n^+)u_n^+ - q_+ G(x,u_n^+)\r) \,\diff \sigma \le M_3,
		\end{split}
	\end{align}
	for all $n\in\N$, with $M_3>0$. Without loss of generality, we can assume $\alpha_-\le\beta_-$ and $\zeta_-\le\theta_-$. From \eqref{h_4}, there exist $\hat{K}_3, \tilde{K}_3, \hat{K}_4, \tilde{K}_4>0$ such that for all $t \in \R$ the following hold
	\begin{align*}
		f(x,t)t - q_+ F(x,t) & \ge \hat{K}_3 |t|^{\alpha_-} - \tilde{K}_3 \quad \text{for a.a.\,}x \in \Omega, \\
		g(x,t)t - q_+ G(x,t) & \ge \hat{K}_4 |t|^{\zeta_-} - \tilde{K}_4 \quad \text{for a.a.\,}x \in \partial\Omega.
	\end{align*}
	Exploiting these relations in \eqref{eq_c}, we derive
	\begin{align*}
		\hat{K}_3 \norm{u_n^+}_{\alpha_-}^{\alpha_-} + \hat{K}_4 \norm{u_n^+}_{\zeta_-,\partial\Omega}^{\zeta_-} \le M_4,
	\end{align*}
	which gives
	\begin{align*}
		\norm{u_n^+}_{\alpha_-} \le M_5
		\quad \text{and} \quad \norm{u_n^+}_{\zeta_-,\partial\Omega} \le \tilde{M}_5 \quad \text{for all }  n \in \N\
	\end{align*}
	for some $M_5, \tilde{M}_5>0$ and Claim 1 is achieved.\\

	{\bf Claim 2:} $\{u_n^+\}_{n \in \N}$ is bounded in $\WH$.\\
	From \eqref{h_1} and \eqref{h_4}, we have that
	\begin{align*}
		\alpha_- < \ell_+ < (p_-)^*
		\quad \text{and} \quad
		\zeta_- < \kappa_+ < (p_-)_*.
	\end{align*}
	Hence, there exist $s,\tau \in (0,1)$ such that
	\begin{align}
		\label{interpolation}
		\frac{1}{\ell_+} = \frac{s}{(p_-)^*} + \frac{1-s}{\alpha_-}
		\quad \text{and} \quad
		\frac{1}{\kappa_+} = \frac{\tau}{(p_-)_*} + \frac{1-\tau}{\zeta_-},
	\end{align}
	and applying the interpolation inequality, see Papageorgiou-Winkert \cite[Proposition 2.3.17 p.116]{Papageorgiou-Winkert-2018}, we obtain
	\begin{align*}
		\norm{u_n^+}_{\ell_+}
		& \le \norm{u_n^+}_{(p_-)^*}^s \norm{u_n^+}_{\alpha_-}^{1-s},\\
		\norm{u_n^+}_{\kappa_+,\partial\Omega}
		& \le \norm{u_n^+}_{(p_-)_*, \partial\Omega}^\tau  \norm{u_n^+}_{\zeta_-, \partial\Omega}^{1-\tau},
	\end{align*}
	for all $n \in \N$.	Taking Claim 1 into account, one has
	\begin{align}\label{eq_d}
		\norm{u_n^+}_{\ell_+} \le M_6 \norm{u_n^+}_{(p_-)^*}^s
		\quad \text{and} \quad \norm{u_n^+}_{\kappa_+,\partial\Omega} \le \tilde{M}_6 \norm{u_n^+}_{(p_-)_*,\partial\Omega}^\tau,
	\end{align}
	for some $M_6, \tilde{M}_6>0$  and for all $n \in \N$. Again, from \eqref{cons_C_2} with $v=u_n^+$, using \eqref{h_1}, it follows that
	\begin{align}
		\label{eq_e}
		\rho(u_n^+) \le \eps_n + K_1 \into \l(|u_n^+| + |u_n^+|^{\ell(x)}\r)\,\diff x + K_2 \intor \l(|u_n^+| + |u_n^+|^{\kappa(x)}\r)\,\diff \sigma.
	\end{align}
	We may assume that $\norm{u_n^+}\geq 1$ for all $n \in \N$, otherwise we are done. Then, using Proposition \ref{properties_modular_norm_complete}(iv), \eqref{eq_e} and \eqref{eq_d}, we derive that
	\begin{align*}
		\norm{u_n^+}^{p_-}
		& \le \rho(u_n^+) \le\eps_n + K_1 \l( \norm{u_n^+}_1 + \norm{u_n^+}_{\ell_+}^{\ell_+}\r) + K_2 \l( \norm{u_n^+}_{1,\partial\Omega} + \norm{u_n^+}_{\kappa_+,\partial\Omega}^{\kappa_+}\r) \\
		& \le \eps_n + M_7 \l( 1+ \norm{u_n^+}_{(p_-)^*}^{s \ell_+}\r) + \tilde{M}_7 \l( 1+ \norm{u_n^+}_{(p_-)_*,\partial\Omega}^{\tau \kappa_+}\r),
	\end{align*}
	with $M_7, \tilde{M}_7>0$. Then, considering the embeddings $\WH\hookrightarrow \Wp{p_-} \hookrightarrow L^{(p_-)^*}(\Omega)$ and $\WH\hookrightarrow W^{1,p_-} (\Omega) \hookrightarrow L^{(p_-)_*}(\partial\Omega)$, we get
	\begin{align*}
		\norm{u_n^+}^{p_-} \le \eps_n + M_8 \l( 1 + \norm{u_n^+}^{s \ell_+} + \norm{u_n^+}^{\tau \kappa_+} \r),
	\end{align*}
	for all $n \in \N$ and for some $M_8 >0$. From \eqref{interpolation}, the definition of $(p_-)^*$ and \eqref{h_4}, one has
	\begin{align*}
		s \ell_+
		& = \frac{(p_-)^* (\ell_+ - \alpha_-)}{(p_-)^*-\alpha_-} =  \frac{Np_- (\ell_+-\alpha_-)}{Np_--N\alpha_-+p_-\alpha_-} \\
		& < \frac{Np_- (\ell_+ -\alpha_-)}{Np_- -N\alpha_- + p_-(\ell_+ - p_-)\frac{N}{p_-}} = p_-.
	\end{align*}
	Similarly, from \eqref{interpolation}, the definition of $(p_-)_*$ and \eqref{h_4}, we have
	\begin{align*}
		\zeta_- > \frac{\zeta_-}{p_-} + (\kappa_+ - p_-)\frac{N-1}{p_- },
	\end{align*}
	which implies
	\begin{align*}
		\tau \kappa_+
		& = \frac{(p_-)_* (\kappa_+ - \zeta_-)}{(p_-)_*-\zeta_-} = \frac{(N-1)p_- (\kappa_+ -\zeta_-)}{(N-1)p_- -N\zeta_- + p_-\zeta_-} \\
		& < \frac{(N-1)p_- (\kappa_+ -\zeta_-)}{(N-1)p_- -N\zeta_- +p_-\l( \frac{\zeta_-}{p_-} + (\kappa_+ - p_-)\frac{N-1}{p_-} \r)} = p_-.
	\end{align*}
	This completes the proof of Claim 2.\\

	{\bf Claim 3:} $u_n \to u$ in $\WH$ up to a subsequence.\\
	From \eqref{u_n_-_to_zero} and Claim 2, it follows that $\{u_n\}_{n \in \N}$ is bounded in $\WH$. Since $\WH$ is a reflexive space, there exists a weakly convergent subsequence in $\WH$, not relabeled, such that
	\begin{align*}
		u_n \weak u \quad \text{in }\WH.
	\end{align*}
	Then, as by \eqref{cons_C_2} in correspondence of $v=u_n-u$, it holds
	\begin{align*}
		\lan I'_+(u_n),u_n-u \ran \to 0 \quad \text{as $n \to \infty$}.
	\end{align*}
	The $f$ and $g$ terms are strongly continuous (see for example \cite[Lemma 4.4]{Crespo-Blanco-Winkert-2022}), hence their limit vanishes and we derive
	\begin{align*}
		\lan A(u_n), u_n-u \ran \to 0 \quad \text{as } n \to \infty.
	\end{align*}
	As $A$ satisfies the $(\Ss_+)$-property, see Proposition \ref{properties_operator_double_phase}, the proof is complete.
\end{proof}

The following results are needed to verify the so-called mountain-pass geometry.

\begin{proposition}
	\label{proposition_lower_estimates_I}
	Let hypotheses \eqref{H2}, \eqref{h_1} and \eqref{h_3} be satisfied. Then, there exist constants $C_i>0, i=1,\ldots,5$ such that
	\begin{align*}
		I(u), I_\pm(u) \ge
		\begin{cases}
			C_1 \norm{u}^{q_+} - C_2 \norm{u}^{\ell_-} - C_3 \norm{u}^{\kappa_-} \quad \text{if } \norm{u} \le \min\{1, C_4, C_5\}, \\
			C_1 \norm{u}^{p_-} - C_2 \norm{u}^{\ell_+} - C_3 \norm{u}^{\kappa_+} \quad \text{if } \norm{u} \ge \max\{1, C_4, C_5\}.
		\end{cases}
	\end{align*}
\end{proposition}
\begin{proof}
	We give the proof only for the functional $I$, the proof for $I_\pm$ is similar. From assumptions \eqref{h_1} and \eqref{h_3} it follows that for all $\eps >0$ there exist $c_\eps, \tilde{c}_\eps >0$ such that
	\begin{equation}
		\label{estimate_3}
		\begin{aligned}
			& |F(x,t)| \le \frac{\eps}{p(x)} |t|^{p(x)} + c_\eps |t|^{\ell(x)} \quad \text{for a.a.\,} x \in \Omega \text{ and for all }t \in \R,\\
			& |G(x,t)| \le \frac{\eps}{p(x)} |t|^{p(x)} + \tilde{c}_\eps |t|^{\kappa(x)} \quad \text{for a.a.\,} x \in \partial\Omega \text{ and for all }t \in \R.
		\end{aligned}
	\end{equation}
	Let $u \in \WH$ be fixed. Using \eqref{estimate_3}, Proposition \ref{properties_norm_modular_r}, the embedding $\WH$ $\hookrightarrow \Lp{\ell(\cdot)}$ with constant $C_{\ell}$ and the embedding $\WH \hookrightarrow L^{\kappa(\cdot)}(\partial\Omega)$ with constant $C_{\kappa,\partial\Omega}$ one has
	\begin{align*}
		I(u)
		&\ge \frac{1}{q_+} \rho_{\mathcal{H}}(\nabla u) + \frac{1}{p_+} \rho_{p(\cdot)}(u) + \frac{1}{p_+} \rho_{p(\cdot),\partial\Omega}(u) \\
		& \quad - \frac{\eps}{p_-} \rho_{p(\cdot)}(u) - c_\eps \rho_{\ell(\cdot)}(u) - \frac{\eps}{p_-} \rho_{p(\cdot),\partial\Omega}(u) - \tilde{c}_\eps \rho_{\kappa(\cdot),\partial\Omega}(u)\\
		& = \frac{1}{q_+} \rho_{\mathcal{H}}(\nabla u) + \l(\frac{1}{p_+}-\frac{\eps}{p_-}\r) \rho_{p(\cdot)}(u) + \l(\frac{1}{p_+}-\frac{\eps}{p_-}\r) \rho_{p(\cdot),\partial\Omega}(u) \\
		&\quad  - c_\eps \rho_{\ell(\cdot)}(u)- \tilde{c}_\eps \rho_{\kappa(\cdot),\partial\Omega}(u)\\
		&\ge \min\l\{\frac{1}{q_+}, \frac{1}{p_+}-\frac{\eps}{p_-}\r\} \rho(u) \\
		&\quad  - c_\eps \max\l\{\norm{u}_{\ell(\cdot)}^{\ell_-}, \norm{u}_{\ell(\cdot)}^{\ell_+}\r\} - \tilde{c}_\eps \max\l\{\norm{u}_{\kappa(\cdot),\partial\Omega}^{\kappa_-}, \norm{u}_{\kappa(\cdot),\partial\Omega}^{\kappa_+}\r\} \\
		&\ge \min\l\{\frac{1}{q_+}, \frac{1}{p_+}-\frac{\eps}{p_-}\r\} \rho(u)\\
		&\quad  - c_\eps \max\l\{ C_{\ell}^{\ell_-} \norm{u}^{\ell_-}, C_{\ell}^{\ell_+} \norm{u}^{\ell_+}\r\} - \tilde{c}_\eps \max\l\{ C_{\kappa,\partial\Omega}^{\kappa_-} \norm{u}^{\kappa_-}, C_{\kappa,\partial\Omega}^{\kappa_+} \norm{u}^{\kappa_+}\r\}.
	\end{align*}
	Choosing $\eps \in \l(0, \frac{(q_+ - p_+)p_-}{p_+ q_+}\r)$ and taking
	\begin{align*}
		C_1 = \frac{1}{q_+}, \quad C_4 = \frac{1}{C_\ell}
		\quad \text{and } \quad
		C_5 = \frac{1}{C_{\kappa, \partial\Omega}},
	\end{align*}
	our statement follows from Proposition \ref{properties_modular_norm_complete}(iii)-(iv) and by setting
	\begin{align*}
		& C_2 = c_\eps C_{\ell}^{\ell_-} \quad \text{and } \quad C_3 = \tilde{c}_\eps C_{\kappa,\partial\Omega}^{\kappa_-} \qquad \text{if } \norm{u} \le \min\{1, C_4, C_5\},  \\
		& C_2 = c_\eps C_{\ell}^{\ell_+} \quad \text{and } \quad C_3 = \tilde{c}_\eps C_{\kappa,\partial\Omega}^{\kappa_+} \qquad \text{if } \norm{u} \ge \max\{1, C_4, C_5\}.
	\end{align*}
\end{proof}

The following result is a direct consequence of Proposition \ref{proposition_lower_estimates_I}.
\begin{proposition}
	\label{proposition_0_local_minimizer}
	Let hypotheses \eqref{H2}, \eqref{h_1} and \eqref{h_3} be satisfied with $q_+ < \ell_-, \kappa_-$. Then
	there exists $\delta>0$ such that
	\begin{align*}
		\inf_{\norm{u} = \delta} I(u)>0
		\quad \text{and} \quad
		\inf_{\norm{u} = \delta} I_\pm(u)>0,
	\end{align*}
	or alternatively, there exists $\lambda>0$ such that $I(u)>0$ for $0<\norm{u}<\lambda$.
\end{proposition}

\begin{proposition}
	\label{proposition_I_-infty}
	Let hypotheses \eqref{H2}, \eqref{h_1} and \eqref{h_2} be satisfied. Then, $I(s u) \to -\infty$ as $s \to \pm\infty$ for every $u \in \WH\setminus \{0\}$. Moreover, $I_\pm(s u) \to -\infty$ as $s \to \pm\infty$ for all $u \in \WH \setminus \{0\}$ such that $u \ge 0$ a.e.\,in $\Omega$.
\end{proposition}
\begin{proof}
	We give the proof only for the functional $I$, since if $u \ge 0$ a.e.\,in $\Omega$ then $I_\pm(s u) = I(s u)$ for $\pm s>0$. Fix $s, \eps \in \R$ and $u \in \WH$ such that $|s|\ge1, \eps \ge 1$ and $u \ne 0$. From \eqref{h_1} and \eqref{h_2} it follows that
	\begin{align*}
		|F(x,t)| \ge \frac{\eps}{q_+} |t|^{q_+} - c_\eps \quad & \text{for a.a.\,}x \in \Omega,\\
		|G(x,t)| \ge \frac{\eps}{q_+} |t|^{q_+} - c_\eps \quad & \text{for a.a.\,}x \in \partial\Omega,
	\end{align*}
	see also \cite[Lemma 4.4]{Crespo-Blanco-Winkert-2022}. Then, using the previous inequalities, one has
	\begin{align*}
		I(s u)
		&\le \frac{|s|^{p_+}}{p_-} \Big(\rho_{p(\cdot)}(\nabla u) + \rho_{p(\cdot)}(u) + \rho_{p(\cdot),\partial\Omega}(u)\Big) + c_\eps \l(|\Omega|+|\partial\Omega|\r) \\
		&\quad  + |s|^{q_+} \l[\frac{\rho_{q(\cdot),\mu}(\nabla u)}{q_-} - \frac{\eps}{q_+}\l(\norm{u}_{q_+}^{q_+} + \norm{u}_{q_+,\partial\Omega}^{q_+}\r)\r] .
	\end{align*}
	Noting that $\norm{u}_{q_+}  < \infty$ and $\norm{u}_{q_+,\partial \Omega}  < \infty$ since $q_+ < l_- < (p_-)^*$ and $q_+ < \kappa_- < (p_-)_*$, we can choose $\eps$ large enough such that the third term is negative and $I(su) \to - \infty$ as $|s| \to \infty$.
\end{proof}

Finally, we state the main result of this section.

\begin{theorem}
	\label{theorem_2_sol}
	Let hypotheses \eqref{H2}, \eqref{h_1}--\eqref{h_4} be satisfied. Then, there exist two nontrivial weak solutions $u_0, v_0 \in \WH \cap \Lp{\infty}$ of problem \eqref{problem}  such that $u_0 \ge 0$ and $v_0\le0$ a.e.\,in $\Omega$.
\end{theorem}
\begin{proof}
	Thanks to Proposition \ref{proposition_Cerami}, \ref{proposition_0_local_minimizer} and \ref{proposition_I_-infty}, we can apply Theorem \ref{mountain_pass_theorem} to both functionals $I_\pm$. Then, there exist $u_0, v_0 \in \WH$ such that $I'_+(u_0)=0$ and $I'_-(v_0)=0$, namely $u_0, v_0$ are weak solutions of problem \eqref{problem}. In particular, from Proposition \ref{proposition_0_local_minimizer} it follows that
	\begin{align*}
		I_+(u_0)\ge \inf_{\norm{u} = \delta} I_ +(u) > 0 = I_+(0),
	\end{align*}
	which implies $u_0 \ne 0$. Analogously, $I_-(v_0)>0$ and $v_0 \ne 0$. Finally, since $\lan I'_+(u_0), v \ran =0$ for every $v \in \WH$, we can choose $v = -u_0^-$ and this leads to
	\begin{align*}
		\rho(-u_0^-) = \into f(x, u_0^+) (- u_0^-) \, \diff x + \intor g(x, u_0^+) (- u_0^-) \, \diff \sigma = 0.
	\end{align*}
	From Proposition \ref{properties_modular_norm_complete} it follows that $-u_0^-=0$ a.e.\,in $\Omega$, hence $u_0 \ge 0$ a.e.\,in $\Omega$. Similarly, we can test $\lan I'_-(v_0), v_0^+ \ran =0$ and derive that $v_0 \le 0$ a.e in $\Omega$. Finally, we know that $u_0$ and $v_0$ are bounded functions because of Theorem \ref{bounded-solutions}.
\end{proof}

\section{Sign changing solution}\label{Section:sign changing}
In this section we present our main result on the existence of a sign-changing solution through the Nehari manifold approach, in addition to the two constant sign solutions obtained in Section \ref{Section:constant sign}. We indicate with $\mathcal{N}$ the Nehari manifold of $I$, defined by
\begin{align*}
	\mathcal{N} = \l\{ u \in \WH \, : \, \lan I'(u),u\ran =0, \, u \ne0\r\}.
\end{align*}
Clearly, any nontrivial weak solution of \eqref{problem} belongs to $\mathcal{N}$, because the weak solutions of \eqref{problem} are exactly the critical points of $I$. Since we are interested in sign-changing solutions, we introduce the following subset of $\mathcal{N}$
\begin{align*}
	\mathcal{N}_0 = \l\{ u \in \WH \, : \, \pm u^\pm \in \mathcal{N}\r\}.
\end{align*}
For an overview on the method of the Nehari manifold, we refer to the book chapter of Szulkin-Weth \cite{Szulkin-Weth-2010}.

First, we prove some properties of the Nehari manifold $\mathcal{N}$ (Proposition \ref{properties_Nehari_manifold}) and of the energy \mbox{functional} $I$ restricted to $\mathcal{N}$ (Proposition \ref{proposition_I_coericive}).

\begin{proposition}
	\label{properties_Nehari_manifold}
	Let hypotheses \eqref{H2}, \eqref{h_1}--\eqref{h_3} and \eqref{h_5} be satisfied. Then, for any $u \in \WH\setminus\{0\}$, there exists a unique $s_u>0$ such that $s_u u \in \mathcal{N}$.\\
	Moreover, one has
	\begin{align*}
		I(s_u u)>0 \quad \text{and} \quad I(s_u u) > I(s u) \quad \text{for all }s>0 \text{ with } s\ne s_u.
	\end{align*}
	and
	\begin{align*}
		\partial_s I(su) > 0\quad \text{for }0 < s < s_u
		\quad \text{and} \quad
		\partial_s I(su) < 0\quad \text{for }s > s_u.
	\end{align*}
\end{proposition}

\begin{proof}
	For any fixed $u \in \WH \setminus\{0\}$ we define $\phi_u\colon [0, \infty ) \to \R$ as follows
	\begin{align*}
		\phi_u(s) = I(su) \quad \text{for all } s \in [0, \infty ).
	\end{align*}
	Clearly, $\phi_u$ belongs to $C\l([0, \infty )\r)$ and $C^1\l((0, \infty )\r)$. From Propositions \ref{proposition_0_local_minimizer} and \ref{proposition_I_-infty} we derive that there exist $\delta, M>0$ such that
	\begin{align} \label{estimate_4}
		\phi_u(s)>0 \quad \text{for } 0<t<\delta \quad \text{and } \quad \phi_u(s)<0 \quad \text{for } t>M.
	\end{align}
	Then, applying the extreme value theorem, we get in particular that $\phi_u$ admits a local maximum, i.e., there exists $0 < s_u \le M$ such that
	\begin{align*}
		\sup_{s \in [0, \infty )} \phi_u(s) = \max_{s \in [0,M]} \phi_u(s) = \phi_u(s_u).
	\end{align*}
	Since $s_u$ is also a critical point of $\phi_u$, in combination with $\phi_u'(s)=\lan I'(su),u\ran$ for every $s>0$, one has
	\begin{align*}
		\phi_u'(s_u) =\lan I'(s_u u),u\ran = 0 \quad \Longrightarrow \quad s_u u \in \mathcal{N}.
	\end{align*}

	\textbf{Claim:} $s_u$ is unique.\\
	From assumption \eqref{h_5} we have that
	\begin{align*}
		& s \mapsto \frac{f(x,su)}{s^{q_+-1}|u|^{q_+-1}} \text{ increasing} \Rightarrow s \mapsto \frac{f(x,su)u}{s^{q_+-1}} \text{ increasing in } \{x \in \Omega\,:\, u(x)>0\},          \\
		& s \mapsto \frac{f(x,su)}{s^{q_+-1}|u|^{q_+-1}} \text{ decreasing} \Rightarrow s \mapsto \frac{f(x,su)u}{s^{q_+-1}} \text{ increasing in } \{x \in \Omega\,:\, u(x)<0\},          \\
		& s \mapsto \frac{g(x,su)}{s^{q_+-1}|u|^{q_+-1}} \text{ increasing} \Rightarrow s \mapsto \frac{g(x,su)u}{s^{q_+-1}} \text{ increasing in } \{x \in \partial\Omega\,:\, u(x)>0\},  \\
		& s \mapsto \frac{g(x,su)}{s^{q_+-1}|u|^{q_+-1}} \text{ decreasing} \Rightarrow s \mapsto \frac{g(x,su)u}{s^{q_+-1}} \text{ increasing in } \{x \in \partial\Omega\,:\, u(x)<0\}.
	\end{align*}
	Multiplying by $1/s^{q_+-1}$ the equation $\phi_u'(s)=\lan I'(su),u\ran=0$ (consider only $s>0$), which is a necessary condition for $su \in \mathcal{N}$, we obtain
	\begin{align*}
			 & \into \l(\frac{|\nabla u|^{p(x)}}{s^{q_+-p(x)}} + \frac{\mu(x)|\nabla u|^{q(x)}}{s^{q_+-q(x)}} \r) \,\diff x + \into \frac{|u|^{p(x)}}{s^{q_+-p(x)}} \,\diff x + \intor \frac{|u|^{p(x)}}{s^{q_+-p(x)}} \,\diff \sigma \\
			 & -\into \frac{f(x,su)u}{s^{q_+-1}} \,\diff x - \intor \frac{g(x,su)u}{s^{q_+-1}} \,\diff \sigma = 0.
	\end{align*}
	As functions of $s$, the left-hand side is strictly decreasing, because it is so in the sets $\{x \in \Omega\,:\, \nabla u \ne 0\}$, $\{x \in \Omega\,:\, u \ne 0\}$ and $\{x \in \partial\Omega\,:\, u \ne 0\}$ and at least decreasing in the rest (recall that $p(x) < q(x) \le q_+$ for all $x \in \close$ and the previous comments for $f$ and $g$). Consequently, there can be at most one single value $s_u>0$ for which the equation holds, namely there exists a unique $s_u>0$ such that $s_u u \in \mathcal{N}$.

	Finally, since $\phi'_u (s)$ has constant sign for $0 < s < s_u$ and $s > s_u$, from \eqref{estimate_4} we can derive
	\begin{align*}
		\phi'_u(s) >0 \quad \text{for } 0<s<s_u \quad \text{and } \quad \phi'_u(s) <0 \quad \text{for } s>s_u.
	\end{align*}
	Thus $s_u$ is a strict maximum for $\phi_u$ and this completes the proof.
\end{proof}

\begin{proposition}
	\label{proposition_I_coericive}
	Let hypotheses \eqref{H2}, \eqref{h_1}--\eqref{h_3} and \eqref{h_5} be satisfied. Then, the functional $I|_{\mathcal{N}}$ is sequentially coercive, namely for any sequence $\{u_n\}_{n \in \N} \subset \mathcal{N}$ such that $\norm{u_n} \xrightarrow{n \to \infty} \infty $ one has $I(u_n) \xrightarrow{n \to \infty} \infty $.
\end{proposition}

\begin{proof}
	Let $\{u_n\}_{n \in \N} \subset \mathcal{N}$ be a sequence such that $\norm{u_n} \xrightarrow{n \to \infty} \infty $ and put
	\begin{align}
		\label{def_y_n}
		y_n = \frac{u_n}{\norm{u_n}} \quad \text{for all $n \in \N$}.
	\end{align}
	Since $\{y_n\}_{n\in\N}$ is bounded in the reflexive space $\WH$ and due to the compact embeddings $\WH\hookrightarrow L^{\ell(\cdot)}(\Omega)$ as well as $\WH\hookrightarrow L^{\kappa(\cdot)}(\partial\Omega)$ (see Proposition \ref{proposition_embeddings}(iii), (v)), there exists a subsequence $\{y_{n_k}\}_{k \in \N}$ and $y \in \WH$ such that
	\begin{align}\label{convergence-properties}
		\begin{split}
			&y_{n_k} \weak y \quad \text{in } \WH,\\
			&y_{n_k} \to y \quad\text{in } \Lp{\ell(\cdot)}  \text{ and pointwisely a.e.\,in }\Omega,\\
			&y_{n_k} \to y \quad\text{in } L^{\kappa(\cdot)}(\partial\Omega)  \text{ and pointwisely a.e.\,in }\partial \Omega.
		\end{split}
	\end{align}

	\textbf{Claim:} $y=0$.\\
	By contradiction, suppose that $y\ne0$. As $\norm{u_n} \to \infty $, there exists $k_0 \in \N$ such that for every $k \ge k_0$ one has $\norm{u_{n_k}}\ge1$ and
	\begin{align*}
		I(u_{n_k})
		& \le \frac{1}{p_-} \rho(u_{n_k}) - \into F(x,u_{n_k})\,\diff x - \intor G(x,u_{n_k})\,\diff\sigma  \\
		& \le \frac{1}{p_-} \norm{u_{n_k}}^{q_+} - \into F(x,u_{n_k})\,\diff x - \intor G(x,u_{n_k})\,\diff\sigma,
	\end{align*}
	where we have used Proposition \ref{properties_modular_norm_complete}(iv). Dividing by $\norm{u_{n_k}}^{q_+}$ and taking \eqref{def_y_n} into account, we obtain
	\begin{align}\label{estimate_1}
		\frac{I(u_{n_k})}{\norm{u_{n_k}}^{q_+}} \le \frac{1}{p_-} - \into \frac{F(x,u_{n_k})}{|u_{n_k}|^{q_+}} |y_{n_k}|^{q_+}\,\diff x - \intor \frac{G(x,u_{n_k})}{|u_{n_k}|^{q_+}} |y_{n_k}|^{q_+}\,\diff\sigma.
	\end{align}
	Now, we observe that if $f$ and $g$ fulfill \eqref{h_1} and \eqref{h_2}, then there exist $M_9,M_{10} >0$ such that
	\begin{equation}\label{estimate_2}
		\begin{aligned}
			& F(x,t)>-M_9 \quad \text{for a.a.\,}x \in \Omega \text{ and for all }t \in \R,\\
			& G(x,t)>-M_{10} \quad \text{for a.a.\,}x \in \partial\Omega \text{ and for all }t \in \R.
		\end{aligned}
	\end{equation}
	Setting $\Omega_0=\l\{x \in \Omega \,:\, y(x)=0\r\}$, by using \eqref{estimate_2}, \eqref{h_2}, \eqref{convergence-properties} and Fatou's Lemma, we get
	\begin{align*}
		&\lim_{k \to \infty} \into \frac{F(x,u_{n_k})}{|u_{n_k}|^{q_+}} |y_{n_k}|^{q_+}\,\diff x\\
		& = \lim_{k \to \infty} \l( \int_{\Omega\setminus\Omega_0} \frac{F(x,u_{n_k})}{|u_{n_k}|^{q_+}} |y_{n_k}|^{q_+}\,\diff x + \int_{\Omega_0} \frac{F(x,u_{n_k})}{\norm{u_{n_k}}^{q_+}} \r)\\
		& \ge \int_{\Omega\setminus\Omega_0} \l(\lim_{k \to \infty} \frac{F(x,u_{n_k})}{|u_{n_k}|^{q_+}} |y_{n_k}|^{q_+}\r)\,\diff x - \lim_{k \to \infty} \frac{M_9 |\Omega_0|}{\norm{u_{n_k}}^{q_+}} \\
		& = \infty.
	\end{align*}
	Analogously, for $\Sigma_0 = \l\{x \in \partial\Omega \,:\, y(x)=0\r\}$, we have
	\begin{align*}
		&\lim_{k \to \infty} \intor \frac{G(x,u_{n_k})}{|u_{n_k}|^{q_+}} |y_{n_k}|^{q_+}\,\diff \sigma\\
		& = \lim_{k \to \infty} \l(\int_{\partial\Omega\setminus\Sigma_0} \frac{G(x,u_{n_k})}{|u_{n_k}|^{q_+}} |y_{n_k}|^{q_+}\,\diff \sigma + \int_{\Sigma_0} \frac{G(x,u_{n_k})}{\norm{u_{n_k}}^{q_+}} \r)\\
		& \ge   \int_{\partial\Omega\setminus\Sigma_0} \l(\lim_{k \to \infty} \frac{G(x,u_{n_k})}{|u_{n_k}|^{q_+}} |y_{n_k}|^{q_+}\r) \,\diff \sigma - \lim_{k \to \infty} \frac{M_{10} |\Sigma_0|}{\norm{u_{n_k}}^{q_+}} \\
		& = \infty.
	\end{align*}
	Hence, passing to the limit as $k \to \infty$ in \eqref{estimate_1}, it follows that
	\begin{align*}
		\lim_{k \to \infty} \frac{I(u_{n_k})}{\norm{u_{n_k}}^{q_+}} = - \infty,
	\end{align*}
	which is a contradiction with $\{u_n\}_{n \in \N} \subseteq \mathcal{N}$ that implies $I(u_n)>0$ for all $n \in \N$ (see Proposition \ref{properties_Nehari_manifold}). Thus, the proof of our claim is complete.

	Recall that $u_{n_k} \in \mathcal{N}$ for every $k \in \N$, from Proposition \ref{properties_Nehari_manifold} it follows that $I(u_{n_k}) \ge I(s u_{n_k})$ for every $s>0, s\ne1$ and for all $k \in \N$. Fixing $s>1$ and using Proposition \ref{properties_modular_norm_complete}(iv), one has
	\begin{align*}
		I(u_{n_k})
		& \ge I(s y_{n_k}) \\
		& \ge \frac{1}{q_+} \rho(s y_{n_k}) - \into F(x,s y_{n_k})\,\diff x - \intor G(x,s y_{n_k}) \,\diff \sigma \\
		& \ge \frac{1}{q_+} \norm{s y_{n_k}}^{p_-} - \into F(x,s y_{n_k})\,\diff x - \intor G(x,s y_{n_k}) \,\diff \sigma \\
		& = \frac{s^{p_-}}{q_+} - \into F(x,s y_{n_k})\,\diff x - \intor G(x,s y_{n_k}) \,\diff \sigma.
	\end{align*}
	Moreover, as a consequence of the assumptions on the nonlinear functions $f$ and $g$, it follows that the integral terms are strongly continuous (see for example \cite[Lemma 4.4]{Crespo-Blanco-Winkert-2022}). Since $s y_{n_k} \weak 0$, we derive that there exists $k_1 \in \N$ such that
	\begin{align*}
		I(u_{n_k}) \ge \frac{s^{p_-}}{q_+} -1 \quad \text{for all $k \ge k_1$}.
	\end{align*}
	From the arbitrariness of $s>1$, we get $I(u_{n_k}) \to \infty$ as $k \to \infty$, which implies that $I(u_n) \xrightarrow{n \to \infty} \infty$ and our statement is achieved.
\end{proof}

Now, we are able to prove the existence of a minimizer of $I$ restricted to $\mathcal{N}_0$.
\begin{proposition}
	Let hypotheses \eqref{H2}, \eqref{h_1}--\eqref{h_3} and \eqref{h_5} be satisfied. Then
	\begin{align*}
		\inf_{u \in \mathcal{N}} I(u) > 0 \quad \text{and} \quad \inf_{u \in \mathcal{N}_0} I(u) > 0.
	\end{align*}
\end{proposition}
\begin{proof}
	Fix $u \in \mathcal{N}$. Then, from Proposition \ref{properties_Nehari_manifold} we have that $I(u) \ge I(su)$ for all $s>0, s \ne 1$. In particular, applying Proposition \ref{proposition_0_local_minimizer}, it follows that
	\begin{align*}
		I(u)\ge I\l(\frac{\delta}{\norm{u}}u\r) \ge \inf_{\norm{u} = \delta} I(u)>0 \quad \text{for all $u \in \mathcal{N}$},
	\end{align*}
	that implies
	\begin{align*}
		\inf_{u \in \mathcal{N}} I(u) > 0.
	\end{align*}
	Now, fix $u \in \mathcal{N}_0$. Since by definition $\pm u^\pm \in \mathcal{N}$, we get
	\begin{align*}
		I(u)= I(u^+)+I(-u^-) \ge 2 \inf_{u \in \mathcal{N}} I(u) > 0  \quad \text{for all $u \in \mathcal{N}_0$},
	\end{align*}
	so we obtain
	\begin{align*}
		\inf_{u \in \mathcal{N}_0} I(u) > 0.
	\end{align*}
\end{proof}

\begin{proposition}
	\label{proposition_existence_minimum}
	Let hypotheses \eqref{H2}, \eqref{h_1}--\eqref{h_3} and \eqref{h_5} be satisfied. Then, there exists $w_0 \in \mathcal{N}_0$ such that
	\begin{align*}
		I(w_0) = \inf_{u \in \mathcal{N}_0} I(u).
	\end{align*}
\end{proposition}
\begin{proof}
	Let $\{u_n\}_{n \in \N} \subseteq \mathcal{N}_0$ be a minimizing sequence, that is, $I(u_n) \searrow \inf_{u \in \mathcal{N}_0} I(u)$. As $u_n \in \mathcal{N}_0$, then $\pm u_n^\pm \in \mathcal{N}$ and $I(\pm u_n^\pm )>0$ for all $n \in \N$ (see Proposition \ref{properties_Nehari_manifold}). Moreover, since $I(u_n) = I(u_n^+) + I(-u_n^-)$ for every $n \in \N$ and from Proposition \ref{proposition_I_coericive}, one has that $\{\pm u_n^\pm\}_{n \in \N}$ are both bounded. Then, by the compact embeddings $\WH\hookrightarrow L^{\ell(\cdot)}(\Omega)$ as well as $\WH\hookrightarrow L^{\kappa(\cdot)}(\partial\Omega)$ (see Proposition \ref{proposition_embeddings}(iii), (v)), there exist subsequences $\{\pm u_{n_k}^\pm\}_{k \in \N}$ and $v_1, v_2 \in \WH$ such that
	\begin{align*}
		 & u_{n_k}^+ \weak v_1,  u_{n_k}^- \weak v_2 \quad \text{in }\WH,\\
		 &u_{n_k}^+ \to v_1, u_{n_k}^- \to v_2\quad\text{in } \Lp{\ell(\cdot)}  \text{ and pointwisely a.e.\,in }\Omega,\\
		 &u_{n_k}^+ \to v_1, u_{n_k}^- \to v_2 \quad\text{in } L^{\kappa(\cdot)}(\partial\Omega)  \text{ and pointwisely a.e.\,in }\partial \Omega,\\
		&\text{with }v_1 \geq 0, v_2\geq 0 \text{ and }v_1v_2=0 \text{ a.e.\,in } \Omega.
	\end{align*}

	\textbf{Claim:} $v_1, v_2 \ne 0$.\\
	Arguing by contradiction, suppose that $v_1=0$. Recalling that $u_{n_k}^+ \in \mathcal{N}$ implies that
	\begin{align*}
		\lan I'(u_{n_k}^+), u_{n_k}^+ \ran =0,
	\end{align*}
	one has
	\begin{align*}
		\rho(u_{n_k}^+) - \into f(x,u_{n_k}^+)(u_{n_k}^+) \,\diff x - \intor g(x, u_{n_k}^+)(u_{n_k}^+) \,\diff\sigma =0.
	\end{align*}
	From the Carathéodory assumption and \eqref{h_1} on the nonlinearities $f$ and $g$, it follows that the two integral terms are strongly continuous (see \cite[Lemma 4.4] {Crespo-Blanco-Winkert-2022}), thus $\rho(u_{n_k}^+) \to 0$ as $k \to \infty$. By Proposition \ref{properties_modular_norm_complete}(v), we get $u_{n_k}^+ \to 0$ in $\WH$ and
	\begin{align*}
		0 < \inf_{u \in \mathcal{N}} I(u) \le I(u_{n_k}^+) \to I(0) = 0 \quad \text{as $k \to \infty$},
	\end{align*}
	that is a contradiction. Analogously we prove that $v_2 \neq 0$ and our claim is true. Now, using Proposition \ref{properties_Nehari_manifold}, there exist $s_1, s_2 >0$ such that $s_1 v_1, s_2 v_2  \in \mathcal{N}$. We put
	\begin{align*}
		w_0 = s_1 v_1 - s_2 v_2 = w_0^+ - w_0^-,
	\end{align*}
	hence $w_0 \in \mathcal{N}_0$. Finally, it remains to prove that $I(w_0) = \inf_{u \in \mathcal{N}_0} I(u)$. It is worth noticing that all the positive terms of $I$ are convex and continuous, thus sequentially weakly lower semicontinuous. On the other hand, we know that the $F$ and $G$ terms are strongly continuous. Hence, $I$ is sequentially weakly lower semicontinuous and this leads to
	\begin{align*}
		\inf_{u \in \mathcal{N}_0} I(u)
		& = \lim_{k \to \infty} I(u_{n_k}) = \lim_{k \to \infty} \l(I(u_{n_k}^+) + I(-u_{n_k}^-)\r) \\
		& \ge \liminf_{k\to \infty} \l(I(s_1 u_{n_k}^+) + I(- s_2 u_{n_k}^-) \r) \\
		& \ge I(s_1v_1) + I(- s_2 v_2)\\
		& = I(w_0^+) + I(-w_0^-) \\
		& = I(w_0) \ge \inf_{u \in \mathcal{N}_0} I(u).
	\end{align*}
	The proof is complete.
\end{proof}

Now, we prove that the minimizer obtained in Proposition \ref{proposition_existence_minimum} is a critical point of the functional $I$.
\begin{proposition}
	\label{proposition_solution_Nehari}
	Let hypotheses \eqref{H2}, \eqref{h_1}--\eqref{h_3} and \eqref{h_5} be satisfied and let $w_0 \in \mathcal{N}_0$ such that $I(w_0) = \displaystyle\inf_{u \in \mathcal{N}_0} I(u)$. Then, $w_0$ is a critical point of the functional $I$.
\end{proposition}

\begin{proof}
	First, we observe something that will be useful in the sequel. Recalling that $\pm w_0^\pm \ne 0$ and indicating with $C_{p_-}$ the constant of the embedding $\WH \hookrightarrow \Lp{p_-}$, we have that
	\begin{align*}
		\norm{ w_0 - v } \geq C_{p_-}^{-1} \norm{ w_0 - v}_{p_-} \geq
		\begin{cases}
			C_{p_-}^{-1} \norm{ w_0^- }_{p_-} & \text{if } v^- = 0,\\
			C_{p_-}^{-1} \norm{ w_0^+ }_{p_-} & \text{if } v^+ = 0,
		\end{cases}
	\end{align*}
	for all $v \in \WH$. Thus, taking
	\begin{align*}
		0 < \delta_0 < \min \left \{ C_{p_-}^{-1} \norm{ w_0^+}_{p_-} , C_{p_-}^{-1} \norm{ w_0^- }_{p_-} \right\},
	\end{align*}
	we have the following implication
	\begin{equation}
		\label{estimate_contradiction}
		\text{if } \norm{ w_0 - v } < \delta_0, \quad \text{then } \quad v^+ \neq 0 \neq v^-.
	\end{equation}
	Now, arguing by contradiction, suppose that $I'(w_0)\ne0$. Then there exist $\gamma, \delta_1 > 0$ such that
	\begin{equation}
		\label{estimate_norm_*}
		\norm{I'(u)}_* \geq \gamma \quad \text{for all } u \in \WH \text{ with } \norm{ u - w_0 } < 3 \delta_1.
	\end{equation}
	Put
	\begin{equation}
		\label{def_delta}
		\delta = \min \l\{ \frac{\delta_0}{2}, \delta_1 \r\}.
	\end{equation}
	From the continuity of the map defined by $(s,t) \mapsto s w_0^+ - t w_0^-$ for every $(s,t) \in [0,\infty)^2$, we have that for every $\delta>0$ there exists $\lambda \in (0,1)$ such that
	\begin{align}\label{estimate_continuity_lambda}
		\norm{ s w_0^+ - t w_0^- - w_0} < \delta,
	\end{align}
	for all $(s,t) \in [0,\infty)^2$ with $\max\{|s-1| , |t-1| \} < \lambda$. Let
	\begin{align*}
		D = (1 - \lambda, 1 + \lambda)^2, \quad
		m_0=\max_{(s,t) \in \partial D} I(s w_0^+ - t w_0^-),
	\end{align*}
	and
	\begin{align}
		\label{def_c}
		c = \inf_{u \in \mathcal{N}_0} I(u).
	\end{align}
	We emphasize that for any $(s,t) \in [0,\infty)^2 \setminus \{(1,1)\}$, using Proposition \ref{properties_Nehari_manifold}, one has
	\begin{equation}
		\begin{aligned}
			\label{estimate_less_inf}
			I(s w_0^+ - t w_0^-)
			& = I(s w_0^+) + I(- t w_0^-)\\
			& < I(w_0^+) + I(- w_0^-) = I(w_0) = \inf_{u \in \mathcal{N}_0} I(u),
		\end{aligned}
	\end{equation}
	which implies that $m_0 < c$.

	In order to use the same notation of the Quantitative Deformation Lemma given in Lemma  \ref{quantitative_deformation_lemma}, we set
	\begin{align*}
		S = B(w_0, \delta), \quad \eps = \min \left\{ \frac{ c - m_0}{4} , \frac{\gamma \, \delta}{8} \right\},
	\end{align*}
	and $\delta, c$ as in \eqref{def_delta} and \eqref{def_c}, respectively. We also notice that by the definition of $S$ it follows that $S_\delta = B(w_0, 2 \delta)$ and $S_{2\delta}=B(w_0,3\delta)$. From \eqref{estimate_norm_*}, we get
	\begin{align*}
		\norm{I'(u)}_* \geq \gamma \geq \frac{8 \eps}{\delta} \quad \text{for all } u \in S_{2\delta},
	\end{align*}
	so all the assumptions of Lemma \ref{quantitative_deformation_lemma} are verified. Hence, there exists a mapping $\eta \in C \l([0,1] \times \WH , \WH\r)$ such that
	\begin{enumerate}
		\item[\textnormal{(i)}]
			$\eta (t, u) = u$, if $t = 0$ or if $u \notin I^{-1}\l([c - 2\eps, c + 2\eps]\r) \cap S_{2 \delta}$,
		\item[\textnormal{(ii)}]
			$I( \eta( 1, u ) ) \leq c - \eps$ for all $u \in I^{-1} \l( ( - \infty, c + \eps] \r) \cap S $,
		\item[\textnormal{(iii)}]
			$\eta(t, \cdot )$ is an homeomorphism of $\WH$ for all $t \in [0,1]$,
		\item[\textnormal{(iv)}]
			$\norm{\eta(t, u) - u} \leq \delta$ for all $u \in \WH$ and $t \in [0,1]$,
		\item[\textnormal{(v)}]
			$I( \eta( \cdot , u))$ is decreasing for all $u \in \WH$,
		\item[\textnormal{(vi)}]
			$I(\eta(t, u)) < c$ for all $u \in I^{-1} \l( ( - \infty, c] \r) \cap S_\delta$ and $t \in (0, 1]$.
	\end{enumerate}
	Afterwards, we consider $h \colon [0,\infty)^2 \to \WH$ defined by
	\begin{align*}
		h(s,t) = \eta ( 1 , s w_0^+ - t w_0^-) \quad \text{ for all } (s,t) \in [0,\infty)^2,
	\end{align*}
	which has the following properties:
	\begin{enumerate}
		\item[\textnormal{(vii)}]
			$h \in C \left( [0,\infty)^2 , \WH \right)$,
		\item[\textnormal{(viii)}]
			$I( h(s,t) ) \leq c - \eps$ for all $(s,t) \in D$, by (ii), \eqref{estimate_continuity_lambda} and \eqref{estimate_less_inf},
		\item[\textnormal{(ix)}]
			$h(D) \subseteq S_\delta$, by (iv) and \eqref{estimate_continuity_lambda},
		\item[\textnormal{(x)}]
			$h(s,t) = s w_0^+ - t w_0^-$ for all $(s,t) \in \partial D$,
	\end{enumerate}
	where the last one follows from (i) and
	\begin{align*}
		I(s w_0^+ - t w_0^-)
		\leq m_0 + c - c
		< c - \left( \frac{c - m_0}{2} \right)
		\leq c - 2 \eps \quad \text{for all } (s,t) \in \partial D.
	\end{align*}
	Now, we define two mappings $H_0,H_1 \colon (0,\infty)^2 \to \R^2$ given by
	\begin{align*}
		& H_0 (s,t) = \left( \; \lan I'(s w_0^+) , w_0^+ \ran \; , \; \lan I'(- t w_0^-) , - w_0^- \ran \; \right), \\
		& H_1 (s,t) = \left( \; \frac{1}{s} \lan I'(h^+(s,t)) , h^+(s,t) \ran \; , \; \frac{1}{t} \lan (- h^- (s,t)) , -h^- (s,t) \ran \; \right),
	\end{align*}
	which are clearly continuous. From Proposition \ref{properties_Nehari_manifold} it follows that
	\begin{equation}
		\begin{aligned}
			\label{estimate_sign_of_I'}
			\lan I'(s w_0^+) , w_0^+ \ran
			\begin{cases}
				> 0 & \text{for all } 0 < s < 1, \\
				< 0 & \text{for all } s > 1,
			\end{cases}
			\\
			\lan I'(- t w_0^-) , - w_0^- \ran
			\begin{cases}
				> 0 & \text{for all } 0 < t < 1, \\
				< 0 & \text{for all } t > 1.
			\end{cases}
		\end{aligned}
	\end{equation}
	Given $A \subseteq \R^n$ open and bounded and $g \in C(A,\R^N)$, we denote by $\deg (g,A,y)$ the Brouwer degree over $A$ of $g$ at the value $y \in \R^N \setminus g(\partial A)$. From the Cartesian product property of the Brouwer degree (see the book of Dinca-Mawhin\cite[Lemma 7.1.1 and Theorem 7.1.1]{Dinca-Mawhin-2021}) we get
	\begin{align*}
		\deg (H_0, D, 0) & = \deg \left( \; \lan I'(s w_0^+) , w_0^+ \ran \; , \; (1 - \lambda, 1 + \lambda) \; , \; 0 \right) \\
		& \quad \times \deg \left( \; \lan I'(- t w_0^-) , - w_0^- \ran \; , \; (1 - \lambda, 1 + \lambda) \; , \; 0 \right),
	\end{align*}
	and by \eqref{estimate_sign_of_I'} and Proposition 1.2.3 of Dinca-Mawhin\cite{Dinca-Mawhin-2021}, we obtain
	\begin{align*}
		\deg (H_0, D, 0) = (-1) (-1) = 1.
	\end{align*}
	We observe that (x) implies $H_0\vert_{\partial D} = H_1\vert_{\partial D}$, so as the Brouwer degree depends on the boundary (\cite[Corollary 1.2.7]{Dinca-Mawhin-2021}), we have
	\begin{align*}
		\deg (H_1, D, 0) = \deg (H_0, D, 0) = 1,
	\end{align*}
	and by the solution property (\cite[Corollary 1.2.5]{Dinca-Mawhin-2021}) it follows that there exists $(s_0 , t_0) \in D$ such that $H_1 (s_0 , t_0) = (0,0)$, namely
	\begin{align*}
		\lan I'(h^+(s_0 , t_0)) , h^+(s_0 , t_0) \ran
		= 0
		= \lan I'(- h^- (s_0 , t_0)) , -h^- (s_0 , t_0) \ran.
	\end{align*}
	Finally, by (ix)
	\begin{align*}
		\norm{h(s_0 , t_0) - w_0} \leq 2 \delta \leq \delta_0,
	\end{align*}
	which, taking \eqref{estimate_contradiction} into account, leads to
	\begin{align*}
		h^+(s_0 , t_0) \neq 0 \quad \text{and} \quad  - h^-(s_0 , t_0) \neq 0.
	\end{align*}
	Thus, $h(s_0 , t_0) \in \mathcal{N}_0$, that is a contradiction with
	\begin{align*}
		I(h(s_0 , t_0)) \leq c - \eps = \inf_{u \in \mathcal{N}_0} I(u) - \eps,
	\end{align*}
	obtained by (viii). This completes the proof.
\end{proof}

Combining Theorem \ref{theorem_2_sol} with Propositions \ref{proposition_existence_minimum} and \ref{proposition_solution_Nehari}, we get the existence of three weak solutions for problem \eqref{problem}. We further know that they are bounded functions thanks to Theorem \ref{bounded-solutions}.

\begin{theorem}
	\label{theorem_3_sol}
	Let hypotheses \eqref{H2} and \eqref{H_{f,g}} be satisfied. Then, there exist three nontrivial weak solutions $u_0,v_0,w_0 \in \WH \cap \Lp{\infty}$ of problem \eqref{problem}  such that $u_0 \geq 0$, $v_0 \leq 0$ and $w_0$ is sign-changing.
\end{theorem}

In the last part of this section, we derive information about the number of nodal domains of the sign-changing solution, that is the number of maximal regions where it has constant sign. The usual definition of nodal domains of a function deals with a continuous function. Nevertheless, we do not know whether our solutions are continuous. Therefore, we use the definition proposed by Crespo-Blanco-Winkert \cite[Section 6]{Crespo-Blanco-Winkert-2022} that we recall in the following.

\begin{definition}
	\label{definition_nodal_domains}
	Let $u \in \WH$ and $A$ be a Borelian subset of $\Omega$ with $|A| > 0$. We say that $A$ is a nodal domain of $u$ if
	\begin{enumerate}[label=(\roman*),font=\normalfont]
		\item
			$u \geq 0$ a.e.\,on $A$ or $u \leq 0$ a.e.\,on $A$;
		\item
			$0 \neq u \mathbbm{1}_A \in \WH$;
		\item
			$A$ is minimal w.r.t.\,\textnormal{(i)} and \textnormal{(ii)}, i.e., if $B \subseteq A$ with $B$ being a Borelian subset of $\Omega$, $|B| > 0$ and $B$ satisfies \textnormal{(i)} and \textnormal{(ii)}, then $|A \setminus B| = 0$.
	\end{enumerate}
\end{definition}

For our purposes, we need to require one more assumption on the nonlinearities:
\begin{enumerate}[label=\textnormal{(h$_{\arabic*}$)},ref=\textnormal{h$_{\arabic*}$}]
	\setcounter{enumi}{5}
	\item
		\label{h_6}
		$f(x,t)t-q_+F(x,t) \geq 0$ and $g(x,t)t-q_+G(x,t)\geq 0$ for all $t\in \mathbb{R}$ and for a.a.\,$x\in\Omega$ and for a.a.\,$x\in\partial\Omega$, respectively.
\end{enumerate}

\begin{proposition}
	\label{proposition_nodal_domains}
	Let hypotheses \eqref{H2}, \eqref{H_{f,g}} and \eqref{h_6} be satisfied. Then, any minimizer of $I\vert_{\mathcal{N}_0}$,  which is also a sign-changing weak solution of problem \eqref{problem}, has exactly two nodal domains.
\end{proposition}
\begin{proof}
	Let $w_0$ be such that $I(w_0) = \displaystyle\inf_{u \in \mathcal{N}_0} I(u)$, fix any $\widetilde{w_0}$ representative of $w_0$ and set
	\begin{align*}
		\Omega_\pm = \l\{ x \in \Omega \,:\, \pm \widetilde{w_0} (x) > 0 \r\}.
	\end{align*}
	As $w_0 \mathbbm{1}_{\Omega_\pm} = \pm \widetilde{w_0} ^\pm$ a.e.\,in $\Omega$, it follows that $\Omega_+$ and $\Omega_-$ satisfy conditions (i) and (ii) of Definition \ref{definition_nodal_domains}. By contradiction, we prove that they are also minimal. We assume, without loss of generality, that there exist Borelian subsets  $A_1, A_2$ of $\Omega$, with $A_1 \cap A_2 = \emptyset,|A_1|>0$ and $|A_2|>0$, such that $\Omega_- = A_1 \dot{\cup} A_2$ and $A_1$ satisfies (i) and (ii) of Definition \ref{definition_nodal_domains}. Moreover, it holds
	\begin{align*}
		 & w_0 \mathbbm{1}_{A_2} = \widetilde{w_0} \mathbbm{1}_{A_2} < 0 \;\; \text{a.e.\,in } A_2, \\
		 & w_0 \mathbbm{1}_{A_2} = w_0 \mathbbm{1}_{\Omega_-} - w_0 \mathbbm{1}_{A_1} \in \WH,
	\end{align*}
	thus $A_2$ also satisfies (i) and (ii). Summarizing, we have
	\begin{align}
		\label{estimate_sign_change}
		\mathbbm{1}_{\Omega_+} w_0 \geq 0, \quad \mathbbm{1}_{A_1} w_0 \leq 0, \quad \mathbbm{1}_{A_2} w_0 \leq 0 \quad \text{a.e.\,in } \Omega,
	\end{align}
	and
	\begin{align*}
		w_0 = \mathbbm{1}_{\Omega_+} w_0 + \mathbbm{1}_{A_1} w_0 + \mathbbm{1}_{A_2} w_0 \quad \text{a.e.\,in }\Omega.
	\end{align*}
	Setting $y_1 = \mathbbm{1}_{\Omega_+} w_0 + \mathbbm{1}_{A_1} w_0$ and $y_2 = \mathbbm{1}_{A_2} w_0$, from \eqref{estimate_sign_change} we have $y_1^+= \mathbbm{1}_{\Omega_+} w_0$ and $- y_1 ^- = \mathbbm{1}_{A_1} w_0$. Since $I'(w_0) = 0$ and as the supports of $y_1^+ , - y_1^-$ and $y_2$ do not overlap, one has
	\begin{align*}
		0 = \lan I'(w_0) , y_1^+ \ran	= \lan I'(y_1^+) , y_1^+ \ran.
	\end{align*}
	Hence $y_1^+ \in \mathcal{N}$ and analogously, $- y_1^- \in \mathcal{N}$. Therefore, $y_1 \in \mathcal{N}_0$. With the same argument one can show that $\lan I'(y_2) , y_2 \ran = 0$. Then, from these properties, we obtain
	\begin{align*}
		I(y_2)
		&= I(y_2) - \frac{1}{q_+} \lan I'(y_2) , y_2 \ran\\
		&\geq  \left( \frac{1}{p_+} - \frac{1}{q_+} \right) \rho_{p(\cdot)} ( \nabla y_2 ) + \left( \frac{1}{p_+} - \frac{1}{q_+} \right) \rho_{p(\cdot)} (y_2 ) +\left( \frac{1}{p_+} - \frac{1}{q_+} \right) \rho_{p(\cdot),\partial\Omega} (y_2 ) \\
		&\quad + \into \left( \frac{1}{q_+} f(x,y_2) y_2 - F(x,y_2) \right) \,\diff x + \intor \left( \frac{1}{q_+} g(x,y_2) y_2 - G(x,y_2) \right) \,\diff \sigma,
	\end{align*}
	which leads to
	\begin{align*}
		I(y_2) > 0,
	\end{align*}
	because of $p_+ < q_+$, $y_2 \neq 0$ and \eqref{h_6}. Finally, we get
	\begin{align*}
		\inf_{u \in \mathcal{N}_0} I(u)	= I (w_0) = I(y_1) + I(y_2) > I(y_1) \ge \inf_{u \in \mathcal{N}_0} I(u),
	\end{align*}
	which is a contradiction and this completes the proof.
\end{proof}
Combining Theorem \ref{theorem_3_sol} and Proposition \ref{proposition_nodal_domains}, we get the main existence result of this paper.

\begin{theorem}
	\label{theorem_nodal_domains}
	Let hypotheses \eqref{H2}, \eqref{H_{f,g}} and \eqref{h_6} be satisfied. Then, there exist three nontrivial weak solutions $u_0,v_0,w_0 \in \WH \cap \Lp{\infty}$ of problem \eqref{problem} such that
	\begin{align*}
		u_0 \geq 0, \quad v_0 \leq 0, \quad w_0 \text{ being sign-changing with two nodal domains.}
	\end{align*}
\end{theorem}

\section*{Acknowledgment}

The first author is member of the Gruppo Nazionale  per l'Analisi Ma\-te\-ma\-ti\-ca, la Probabilit\`{a} e le loro Applicazioni  (GNAMPA) of the Istituto Nazionale di Alta Matematica (INdAM). The second author was funded by the Deutsche Forschungsgemeinschaft (DFG, German Research Foundation) under Germany's Excellence Strategy - The Berlin Mathematics Research Center MATH+ and the Berlin Mathematical School (BMS) (EXC-2046/1, project ID: 390685689). The first author thanks the University of Technology Berlin for the kind hospitality during a research stay in February-May 2023.

%
%
%
%
%


\end{document}